\documentclass[a4paper,11pt]{article}

\usepackage{amsmath,amssymb,amsthm}

\usepackage{amsfonts,amssymb}

\numberwithin{equation}{section}

\newtheorem{theorem}{Theorem}[section]
\newtheorem{proposition}[theorem]{Proposition}
\newtheorem{lemma}[theorem]{Lemma}

\newtheorem{remark}[theorem]{Remark}

\allowdisplaybreaks

\begin{document}
 \title{Construction of  Multi-Bubble Solutions for a System of Elliptic Equations arising in Rank Two Gauge Theory }
  \author{Hsin-Yuan Huang  }
  \date{\small\it  Department of Applied Mathematics, National Chiao-Tung University,\\ Hsinchu, Taiwan}
 \maketitle
 \begin{abstract}
We study the existence of multi-bubble solutions for the following skew-symmetric Chern--Simons system
\begin{equation}\label{e051}
\left\{
\begin{split}
&\Delta u_1+\frac{1}{\varepsilon^2}e^{u_2}(1-e^{u_1})=4\pi\sum_{i=1}^{2k}\delta_{p_{1,i}}\\
&\Delta u_2+\frac{1}{\varepsilon^2}e^{u_1}(1-e^{u_2})=4\pi\sum_{i=1}^{2k}\delta_{p_{2,i}}
\end{split}
\text{  in  }\quad \Omega\right.,
\end{equation}
where  $k\geq 1$ and $\Omega$ is a flat tours in $\mathbb{R}^2$. It  continues the joint work with Zhang\cite{HZ-2015}, where we obtained the necessary conditions for the existence of bubbling solutions of Liouville type.  Under nearly necessary conditions(see Theorem \ref{main-thm}), we show that there exist a  sequence of  solutions  $(u_{1,\varepsilon}, u_{2,\varepsilon})$ to \eqref{e051} such that $u_{1,\varepsilon}$ and $u_{2,\varepsilon}$ blow up simultaneously at $k$ points in $\Omega$ as $\varepsilon\to 0$.

 \end{abstract}

\section{Introduction }

Over  the past few decades,  various Chern--Simons models have been proposed in the study of  condensed matter physics and particle physics, including the relativistic Chern-Simons models of high temperature superconductivity \cite{JW,Dunne1}, Lozano-Marqu\'es-Moreno-Schaposnik model \cite{LMMS} of bosonic sector of  $\mathcal{N}=2$ supersymmetric  Chern-Simons-Higgs theory,   Gudnason model \cite{Gu1, Gu2} of $\mathcal{N}=2$ supersymmetric  Yang-Mills-Chern-Simons-Higgs theory, Aharony--Bergman--Jafferis--Maldacena model\cite{abjm} and so on. The self-dual solutions of these Chern-Simons models often can  be reduced to systems of elliptic partial differential equations with exponential nonlinearities. We refer the readers to the  \cite{Dunne1,yangbook,Tbook} for exhaustive bibliography.\par

For the Abelian Chern-Simons models, Hong, Kim and Pac\cite{HKP}, and Jackiw and Weinberg\cite{JW}  considered a  model with one Chern--Simons gauge field and constructed selfdual Abelian Chern--Simons--Higgs vortices which   describe anyonic solitons in 2+1 dimensions.   Speilman et al.\cite{SFEG} observed no parity breaking in an experiment with high temperature superconductivity. In  \cite{Hagen} and   \cite{Wil}, those authors   indicated the parity broken  may not  happen in the a field theory with even number of Chern-Simons gauge fields.  One of the simplest models of this kind is the $U(1)\times U(1)$ Chern-Simons model of two Higgs fields, where each of them coupled to one of two Chern--Simons fields.  In this paper, we will investigate  the relativistic self-dual  $U(1)\times U(1)$ Chern-Simons model proposed by Kim et al\cite{KLKLM}.    We give  a brief description on this model below.\par
Consider the Lagrangian action  density of the   $[U(1)]^2$ Chern-Simons model with only mutual Chern--Simons interaction  is  given in the form
\begin{equation}
 \mathcal{L}  =  -\frac{\varepsilon}{ 4}\epsilon^{\mu\nu\alpha}\left(A^{(1)}_\mu F^{(2)}_{\mu\nu}+A^{(2)}_\mu F^{(1)}_{\mu\nu}\right)+\sum\limits_{i=1}^2D_\mu\phi_i\overline{D^\mu\phi_i}-V(\phi_1, \phi_2), \label{a1}
\end{equation}
where $\varepsilon>0$ is a coupling parameter, $(A_\mu^{(1)})$ and $(A_\mu^{(2)})$   are two
associated Abelian gauge fields with the electromagnetic fields $F^{(i)}_{\mu\nu}=\partial_\mu A_\nu^{(i)}-\partial_\nu A_\mu^{(i)}$,
$\phi_1$  and $\phi_2$ are two Higgs scalar fields with the  covariant derivatives $D_\mu\phi_i=\partial_\mu\phi_i-\mathrm{i} A^{(i)}_\mu\phi_i$($ \mu=0,1, 2,\,i=1,2 $), and the Higgs potential $ V(\phi_1,\phi_2) $ is  taken as
 \begin{equation}
  V(\phi_1,\phi_2) =    \frac{1}{4\varepsilon^2}\left(|\phi_2|^2\left[|\phi_1|^2-1\right]^2+|\phi_1|^2\left[|\phi_2|^2-1\right]^2\right).
\end{equation}

Set $u_{i,\varepsilon}=\ln |\phi_i|^2$, and denote the zeros of $\phi_i$ by $\{p_{i,1}, \dots, p_{i,{N_i}}\}$, $i=1, 2$ as in \cite{JT}. After a BPS reduction \cite{Bo,PS}, we find that  $(u_{1,\varepsilon},u_{2,\varepsilon})$ satisfies
\begin{equation}\label{e001}
\left\{
\begin{split}
\Delta u_1+\frac{1}{\varepsilon^2}e^{u_2}(1-e^{u_1})=4\pi\sum_{i=1}^{N_1}\delta_{p_{1,i}}\\
\Delta u_2+\frac{1}{\varepsilon^2}e^{u_1}(1-e^{u_2})=4\pi\sum_{i=1}^{N_2}\delta_{p_{2,i}}
\end{split}
\text{  in  }\quad \Omega\right.,
\end{equation}
where $\delta_p$ is  the Dirac measure   at $p$.  See \cite{KLKLM,dz94,LPY} for the details of the derivation of \eqref{e001}. In physical literature, $\Omega$ here is usually refereed to  $\mathbb{R}^2$ or a flat tours  in $\mathbb{R}^2$. In this paper, we consider the case of flat tours.  We refer the readers to \cite{LPY,LP,CCL,HL1,HHL1,CY2014} and reference therein for the recent developments.\par

Set $u_1\equiv u_2$ and $\{p_i^1\}_{i=1}^{N_1}=\{p_j^2\}_{j=1}^{N_2}$,  then the system \eqref{e001} is reduced to the Abelian Chern--Simons equation with one Higgs particle   proposed by  Kim--Pac\cite{HKP} and Jackiw--Weinberg\cite{JW},
\begin{equation} \label{eq17}
 \Delta u+\frac{1}{\varepsilon^2}e^u(1-e^u)=4\pi\sum_{s=1}^{N}\delta_{p_s},
\end{equation}
 which has been extensively studied for  more than  twenty years. We refer the readers to\cite{SY,T,CY1,C,Dunne1,T1,CFL,choe,CKL,DEFM} and reference therein for more details.\par

We introduce the Green function $ G(x,p)$ to  eliminate the Dirac measure in \eqref{eq17}.
Here, the Green function $G(x,p)$ is a doubly periodic function on $\partial \Omega$ and satisfies
\begin{equation}
-\Delta G(x,p)=\delta_p-\frac{1}{|\Omega|},
\end{equation}
and $ |\Omega|$ is the measure of $\Omega$.
 Let $u_0(x)=-4\pi\sum_{j=1}^N G(x, p_j)$. With the transformation $u\to u+u_0$, \eqref{eq17} becomes
 \begin{equation} \label{eq18}
 \Delta u+\frac{1}{\varepsilon^2}e^{u+u_0}(1-e^{u+u_0})= \frac{4N\pi}{|\Omega|}.
\end{equation}
Choe and Kim\cite{CK}  showed that, there may have a sequence of solutions  to \eqref{eq18} satisfying the following:\\
{\it There is finite set  $\{x_{1,\varepsilon},\cdots, x_{l,\varepsilon}\}$, $x_{j,\varepsilon}\in \Omega$, $j=1,\cdots,l$, such that, as $\varepsilon\to 0$
\begin{equation}\label{eq19}
u_{\varepsilon}(x_{j,\varepsilon})+2\ln \frac{1}{\varepsilon}\to\infty,\,\,j=1,\cdots, l,
\end{equation}
and
\begin{equation}\label{eq20}
u_{\varepsilon} +2\ln \frac{1}{\varepsilon}\to-\infty \text{ uniformly on any compact subset of   }\Omega\setminus\{q_1,\cdots, q_l\},
\end{equation}
where $q_j=\lim_{\varepsilon\to 0}x_{j,\varepsilon}$. Furthermore,
\begin{equation}
\frac{1}{\varepsilon^2}e^{u_{\varepsilon}+u_0}(1-e^{u_{\varepsilon}+u_0})\to \sum_{j=1}^l M_j\delta_{q_j}\quad\text{ as }\quad\varepsilon\to 0,\, M_j\geq 8\pi.
\end{equation}}

Solutions of \eqref{eq18} satisfying \eqref{eq19} and \eqref{eq20} are called {\it bubbling solutions}  and $q_j$ is called the blow-up point of the bubbling solution.\par
 We can classify these blow-up points as follows:  After suitable rescaling at the blow-up point $q_j$,     the bubbling solution $u_{\varepsilon}$ to \eqref{eq18}  converges to either the entire solution of
\begin{equation}\label{eq41}
\Delta u+|x|^{2m}e^{u}(1-|x|^{2m}e^u)=0,
\end{equation}
or   the entire solution of
\begin{equation}\label{eq42}
\Delta u+|x|^{2m}e^{u}=0.
\end{equation}
Here, $m=0$ if $q_j$ is not a vortex point and $m=\#\{p_i: p_i=q_j\}$ if $q_j$ is a vortex point. The blow-up point is called Chern-Simons type if the limiting equation is \eqref{eq41} and is called  mean field type if the limiting equation is \eqref{eq42}.
The existence and non-existence of bubbling solutions of \eqref{eq18} have been studied in a series of work of Lin and Yan\cite{LY,LY2,LY3}.\par

For \eqref{e001},  we may use the following system to construct the bubbling solutions to \eqref{e001} which blow up at the same point $q.$
\begin{itemize}
\item[(i)]  \begin{equation}\label{eq36} \text{ Chern-Simons system:  }
\left\{
\begin{split}
\Delta u_1+ |x|^{m_2}e^{u_2}(1-|x|^{m_1}e^{u_1})=0\\
\Delta u_2+  |x|^{m_1}e^{u_1}(1-|x|^{m_2}e^{u_2})=0
\end{split}\quad
\text{  in  }\quad \mathbb{R}^2\right..
\end{equation}
\item[(ii)]
\begin{equation}\label{eq37} \text{ Liouville system:   }
\left\{
\begin{split}
\Delta u_1+ |x|^{m_2}e^{u_2} =0\\
\Delta u_2+ |x|^{m_1}e^{u_1} =0
\end{split}\quad
\text{  in  }\quad \mathbb{R}^2\right..
\end{equation}
\end{itemize}
Here, $m_1=\#\{p_{1,j}:\, p_{1,j}=q\}$ and $m_2=\#\{p_{2,j}:\, p_{2,j}=q\}$. For the case $(i)$, the blow-up point is called the Chern-Simons type; for the case $(ii)$, the blow-up point is called the Liouville type. \par

For the existence of the bubbling solutions of Chern-Simons type to \eqref{e001}, there are two known results. The first one is due to Lin and Yan\cite{LY}, in which they assume that $N_1=N_2$ and one of the vortex points $\{p_{1,i}\}_{i=1,\cdots, N_1}$ and $\{p_{1,i}\}_{i=1,\cdots, N_1}$ coincide, then they  constructed a Chern-Simons type bubbling solution blow up at that vortex point.

{\bf Theorem A. }\cite{LY} {\it
  Suppose that $N_1=N_2=N>4$ and $p_{1,1}=p_{2,1}=p_1$ and $p_{1,1}\neq p_{1,j}$, $p_{2,j},$ $j>1$. Then, there is an $\varepsilon_0>0$, such that for any $\varepsilon\in(0, \varepsilon_0)$, \eqref{e001} has a solution $(u_{1, \varepsilon}, u_{2, \varepsilon})$, satisfying
\begin{equation}
  u_{1, \varepsilon}-u_{2, \varepsilon}\to 4\pi \sum\limits_{j=2}^{N}G(p_{1,1}, p_{1,j})-4\pi \sum\limits_{j=2}^{N}G(p_{2,1}, p_{2,j})
 \end{equation}
  and
  \begin{equation}
   \frac{1}{\varepsilon^2}e^{ u_{i, \varepsilon} }\left(1-e^{ u_{j, \varepsilon} }\right)\to 4\pi N\delta_{p_1},\quad i\not=j,\quad i,j\in\{1,2\}
  \end{equation}
   as $\varepsilon\to 0$.
 }\par
   The idea in the proof of  Theorem A is using the ansatz
$$u_1-u_2\thickapprox 4\pi \sum\limits_{j=2}^{N}G(p_{1,1}, p_{1,j})-4\pi \sum\limits_{j=2}^{N}G(p_{2,1}, p_{2,j})$$ to reduce the system\eqref{e001} to a single equation.\par
In a joint work with Han and Lin\cite{HHL1}, we proved the existence of Chern-Simons type bubbling solutions when the blow-up point is not vortex point.\\
\\
{\bf Theorem B. }\cite{HHL1} {\it
  Suppose $(N_1-1)(N_2-1)>1$, and $q$  satisfies
  \begin{equation}
  D ( N_1u_{2, 0}+N_2u_{1, 0})(q)=0   \label{e4}
 \end{equation}
 and
 \begin{equation}deg(D(N_1 u_{2,0}+N_2u_{1,0}), q)\neq 0. \end{equation}
 Then  the system \eqref{e001} admits Chern-Simons type bubbling solutions $(u_{1, \varepsilon}, u_{2, \varepsilon})$ blowing up at $q$.}\\
 \par

Theorem A and B  only discussed the one blow-up case. However, the problem on the existence of multi-bubble solutions of \eqref{e001} is still open. The analysis of linearized system to \eqref{eq36} and \eqref{eq37} is not fully understood when $(m_1,m_2)\not=(0,0)$. Thus,   we only investigate the issue on the existence of the {\it bubbling solutions of Liouville type }to \eqref{e001} whose blow up points  are regular points. We give more precise description on this type bubbling solutions below. \par
Let $(u_{1,\varepsilon},u_{2,\varepsilon})$ be a sequence of solutions to \eqref{e001} blowing up at $\{q_1,\cdots, q_k\}$ and $q_i\not\in\{p_{1,i}, p_{2,j}\}_{1\leq i\leq N_1,\,1\leq  j\leq N_2}$. For a small constant $d>0$, we define the local mass of $(u_{1,\varepsilon},u_{2,\varepsilon})$ at $q_i$:  $$(m_{1,i,\varepsilon},m_{2,i,\varepsilon})=\left(\frac{1}{\varepsilon^2}\int_{B_d(q_i)}e^{u_{2,\varepsilon}}(1-e^{u_{1,\varepsilon}}), \frac{1}{\varepsilon^2}\int_{B_d(q_i)}e^{u_{1,\varepsilon}}(1-e^{u_{2,\varepsilon}}) \right).$$
We assume that  there exist
$\{x_{1,j,\varepsilon}\}_{j=1}^k$ and $\{x_{2,j,\varepsilon}\}_{j=1}^k$ such that
\begin{itemize}
\item[($a_1$)]  $u_{i,\varepsilon}(x_{i,j,\varepsilon})+2\ln \frac{1}{\varepsilon} \to +\infty $  as $\varepsilon\to 0$, $i=1,2$, $j=1,\cdots,k$.
\item[($a_2$)]   $u_{i,\varepsilon}(x)+2\ln \frac{1}{\varepsilon} \to -\infty $  as $\varepsilon\to 0$ uniformly on any compact set of $\Omega\setminus\{ q_1,\cdots,q_k\}$, where
$q_j=\lim_{\varepsilon\to 0}x_{1,j,\varepsilon}=\lim_{\varepsilon\to 0}x_{2,j,\varepsilon}$, $j=1,\cdots,k$.

\item[($a_3$)] $\beta_{j,\varepsilon}=\max\{ u_{1,\varepsilon}(x_{1,j,\varepsilon}),u_{2,\varepsilon}(x_{2,j,\varepsilon})  \}\to -\infty$ as  $\varepsilon\to 0$.
\item[($a_4$)]  $ |u_{1,j,\varepsilon}(x_{1,j,\varepsilon})-u_{2,\varepsilon}(x_{2,j,\varepsilon})|=O(1)$
\item[($a_5$)]   $|x_{i,j,\varepsilon}-q_j|<C\varepsilon e^{-\frac{1}{2}\beta_{j,\varepsilon}}$ for some constant $C>0$.
\end{itemize}
The solutions $(u_{1,\varepsilon},u_{2,\varepsilon})$ satisfying $(a_1)-(a_5)$ are called {\it fully bubbling solutions of Liouville type}.\par Let $\mu_{j,\varepsilon}=\varepsilon e^{-\frac{1}{2}\beta_{j,\varepsilon}}$ and assume $\max\{u_{1,\varepsilon}(x),u_{2,\varepsilon}(x)\}$ attains its maximum at $x_{j,\varepsilon}$ for $x$ near $q_j$. Formally,
$$(\tilde{u}_{1}(y),\tilde{u}_2(y))=(u_{1,\varepsilon}(\mu_{j,\varepsilon}y+x_{j,\varepsilon})-\beta_{j,\varepsilon},u_{2,\varepsilon}( \mu_{j,\varepsilon}y+x_{j,\varepsilon})- \beta_{j,\varepsilon})$$
converges to the entire solution $(U_{1,j},U_{2,j})$ of \eqref{eq37} with  $m_1=m_2=0$. Furthermore, the flux $(M_{1,j},M_{2,j})=(\int_{\mathbb{R}^2} e^{U_{2,j}},\int_{\mathbb{R}^2} e^{U_{1,j}} )$  satisfies
\begin{equation}\label{eq60}
\frac{1}{M_{1,j}}+\frac{1}{M_{2,j}}=\frac{1}{4\pi}.
\end{equation} Thus,
either $\min\{ M_{1,j}, M_{2,j}\}< 8\pi$ or $M_{1,j}=M_{2,j}=8\pi$.    The analysis of bubbling solutions of these two types has different phenomena. We refer the readers to \cite{ZL-JFA} for the related bubbling analysis for Liouville system.\par

In a joint work with Zhang\cite{HZ-2015}, we  obtained necessary conditions for the fully bubbling solutions of  Liouville type  to \eqref{e001}.  Before we state the main result in \cite{HZ-2015}, we introduce some notations.  Let ${\bf q}=(q_1,\cdots,q_k)$, $q_i\in\mathbb{R}^2$, $i=1,\cdots,k$,
$$
 G_1^*({\bf q})=\sum_{i=1}^k u_{0,1}(q_i)+8\pi \sum_{ 1\leq i<j\leq k}G(q_i,q_j)
\,\,\text{ and  }\,\,
 G_2^*({\bf q})=\sum_{i=1}^k u_{0,2}(q_i)+8\pi \sum_{ 1\leq i<j\leq k}G(q_i,q_j).
$$ Denote the function $f_{i,j,\textbf{q}}$($i=1,2,$, $j=1,\cdots,k$) as follows.
\begin{equation}
f_{i,j,\textbf{q}}= 8\pi(\gamma(y,q_j)-\gamma(q_j,q_j))+\sum_{l\not = j}(G(y,q_l)-G(q_j,q_l))+u_{0,i}(y)-u_{0,i}(q_j),
\end{equation}
 where $\gamma(y,q)$  is the regular part of $G(y,p)$. So, it is clear that
 $\frac{\partial G_1^*({\bf x})}{\partial x_{j,h}}=\frac{\partial f_{1,j,x}(x)}{\partial x_h}$
 and
 $\frac{\partial G_2^*({\bf x})}{\partial x_{j,h}}=\frac{\partial f_{2,j,x}(x)}{\partial x_h}$.
We define the quantity $\mathcal{D}^{(2)}({\bf q})$ as follows
\begin{equation}\label{e052}
\begin{split}
\mathcal{D}^{(2)}({\bf q})&=\lim_{\delta\to 0} \left( \sum_{j=1}^k  \frac{\rho_j}{\rho_1}
 \left(  \int_{\Omega_j\setminus B_{\delta}  } \frac{e^{f_{1,j,{\bf q}}}-1}{|x-q_{j }|^{4}}-  \int_{\mathbb{R}^2\setminus \Omega_j  } \frac{1}{|x-q_{j }|^{4}}  \right)\right.\\
 &\qquad +  \left.  \sum_{j=1}^k \frac{\rho_j^*}{\rho_1^*}\left( \int_{\Omega_j\setminus B_{\delta}  } \frac{e^{f_{2,j,{\bf q}}}-1}{|x-q_{j }|^{4}}-  \int_{\mathbb{R}^2\setminus \Omega_j  } \frac{1}{|x-q_{j}|^{4} }  \right)\right)
\end{split}
\end{equation}
where  $\{\Omega_j\}_{j=1,\cdots,k}$ is any open set with
\begin{itemize}
\item[(1)] $\Omega_i\cap \Omega_j=\emptyset$, $i\not =j$, $i,j\in\{1,\cdots,k\}$,
\item[(2)] $\cup_{j=1}^k \overline{\Omega_j}=\overline{\Omega}$,
\item[(3)] $B_{d_j}(q_{j})\subset\subset \Omega_j$, $j=1,\cdots,k$,
\end{itemize}
$\rho_j=e^{8\pi(\gamma(q_j,q_j)+\sum_{l\not =j}G(q_j,q_l)  ) +u_{0,1}(q_j)      }    ,$  and  $\rho^*_j=e^{8\pi(\gamma(q_j,q_j)+\sum_{l\not =j}G(q_j,q_l)  ) +u_{0,2}(q_j)      }   .$ \par
In particular, when $k=1$,
\begin{equation}
\mathcal{D}^{(2)}(q) =\lim_{\delta\to 0}
 \left(  \int_{\Omega \setminus B_{\delta}  } \frac{e^{ u_{0,1}(x)-u_{0,1}(q)  }-1}{|x-q|^{4}} + \frac{e^{u_{0,1}(x)-u_{0,1}(q) }-1}{|x-q|^{4}}-  \int_{\mathbb{R}^2\setminus \Omega } \frac{2}{|x-q |^{4} }  \right) .
\end{equation}\par
When  $\Omega$ is a rectangle, $k=1$ and $p_{i,j}=p$( $i,j=1,2$), it was shown \cite{linwang2010} that the Green function $G(x,p)$ has three critical points: two of them are saddle points whose corresponding $\mathcal{D}^{(2)}(q)>0$, and
the other is a maximum points whose $\mathcal{D}^{(2)}(q)<0$. We refer to \cite{linwang2010,linwang20171} for the related discussion on this kind quantity. \par

The necessary conditions for the existence of fully bubbling solutions of Liouville type are given as follows. \\
\\
{\bf Theorem C} \cite{HZ-2015} Suppose $(u_{1,\varepsilon}, u_{2,\varepsilon})$ is a sequence of fully bubbling solution of Liouville type to \eqref{e001} with $N_1=N_2=2k$ and blow up set is
$ \{q_1,\cdots, q_k \} \not \in \{p_{1,i}, p_{2,i}\}_{i=1,\cdots,2k}$.
 Then
\begin{itemize}
\item[$(1)$] $(u_{0,1}-u_{0,2})(q_i)=(u_{0,1}-u_{0,2})(q_j)$, $1\leq i,j\leq k.$
\item[$(2)$] ${\bf q}=(q_1,\cdots,q_k)$ is  a critical point of $G_1^*$ and $G_2^*$.
\item[$(3)$] $\mathcal{D}^{(2)}({\bf q})\leq 0$.
\end{itemize} \par

Naturally, we are led to the question whether the conditions obtained in Theorem C are    sufficient for the existence of such solutions.  In this paper, we will construct a sequence of bubbling solutions whose limiting local masses are $(8\pi,8\pi)$ and the sufficient condition are nearly necessary. Since we assume that $N_1=N_2=2k$, it is natural to use the entire solution $(U_1,U_2)$ of \eqref{eq37} with flux $(M_1,M_2)\approx (8\pi,8\pi)$ to construct the approximation solutions for \eqref{e001}. But $(M_1,M_2)$ must satisfy
\eqref{eq60}, it adds a difficulty in construction of good approximation solutions for multi-bubble case. With the assumption (A.1)(see below), we can choose $(U_1,U_2)$ with $(\int_{\mathbb{R}^2} e^{U_2}dx,\int_{\mathbb{R}^2} e^{U_1}dx )=(8\pi,8\pi)$ to construct the approximation solutions around each blow-up point and glue them together.   It is a crucial step when we apply the invertibility of linear operator in the contraction mapping argument.  \par

\begin{theorem}\label{main-thm}
Assume that
\begin{itemize}
\item[$($A.$1)$] $(u_{0,1}-u_{0,2})(q_i)=(u_{0,1}-u_{0,2})(q_j)$, $1\leq i,j\leq k.$
\item[$($A.$2)$] ${\bf q}$ be a critical point of $G_1^*$ and $G_2^*$
\item[$($A.$3)$] ${\bf q}$ be non-degenerate critical point of $G_1^*+G_2^*$
\item[$($A.$4)$] $\mathcal{D}^{(2)}({\bf q})<0$
\end{itemize}
Then for $\varepsilon>0$ small,
\eqref{e001} has a solution $(u_{1,\varepsilon},u_{2,\varepsilon})$ satisfies
\begin{equation}
\frac{1}{\varepsilon^2}e^{u_i}(1-e^{u_j})\to 8\pi \sum_{k=1}^N\delta_{q_l},\quad i\not= j, i,j=1,2,\,\,l=1,\cdots,k,\quad\text{as}\quad \varepsilon\to 0.
\end{equation}
\end{theorem}
When $k=1$, Theorem \ref{main-thm} is reduced to the following theorem.
\begin{theorem}\label{thm2}
Assume $q$ is a critical point of $u_{0,1}$ and $u_{0,2}$  and a non-dgenerate critical point of $u_{0,1}+u_{0,2}$,  and $\mathcal{D}^{(2)}(q)<0$, then
there exists a sequence of solution $(u_{1,\varepsilon},u_{2,\varepsilon})$
to \eqref{e001} blowing up at $q$.
\end{theorem}


It will be interesting to consider the Liouville type  bubbling  solutions with $N_1\not =N_2$. In this case, $\frac{k}{N_1}+\frac{k}{N_2}=1$ and $k\geq 2$. For instance, $k=2$ and $(N_1,N_2)=(3,6)$.   As suggested in the study of Liouville system\cite{ZL-JFA},  the quantity of $\mathcal{D}^{(2)}({\bf q})$ should be a slight different form and only depend on $\min\{\frac{4\pi N_1}{k},\frac{4\pi N_2}{k}\}$.  In \cite{ZL-JFA}, they only consider the simple bubble case, but in our case, it must be multi-bubble.  Furthermore, the invertibility of the corresponding linear operator could be another difficulty.  Since the height of each bubble is compatible, the dimension of the  kernel space to the  corresponding linear operator is $2k+1$. The kernel with respect to height is not a local condition which makes the analysis of linear operator difficult.   We will come back to this issue in the  future.\par

The rest of our paper is organized as follows. In  Sec. 2, we construct  approximate solutions of \eqref{e001}. In Sec. 3, we use the contraction mapping theorem to prove the existence of multi-bubble solutions. Appendix is devoted to   the invertibility of the linear operator $L_{\mu}$.

\section{The Approximation Solution}

In this section, we will construct an approximation solution $(U_{1,\mu}, U_{2,\mu})$(see \eqref{eq04} below) for \eqref{e001} with the assumptions in Theorem \ref{main-thm} and give the estimates on the approximation solution. This construction follows the method in \cite{LY2}. Without loss of generality, we assume that $|\Omega|=1$.\par

We consider the solution of mean field equation
\begin{equation}\label{eq62}
V_{x_i,\mu_i}(y)=\ln\frac{8\mu_i^2}{(1+\mu_i^2|y-x_i|^2)^2},\,\,x_i\in\mathbb{R}^2,\,\,\mu_i>0,
\end{equation}
which satisfies

\begin{equation}\left\{
\begin{split}
&\Delta V_{x_i,\mu_i}(y)+e^{V_{x_i,\mu_i}(y)}=0\text{     in   }\mathbb{R}^2,\\
&\int_{\mathbb{R}^2}e^{V_{x_i,\mu_i}(y)} dy=8\pi.
\end{split}
\right.
\end{equation}

In the construction of approximation solution, we always assume that $|x_i-q_i|<\frac{C_1}{\mu_i}$ for some large $C_1$
 and $$\mu_i\in \left[\frac{\beta_1}{\sqrt{\varepsilon}},\frac{\beta_2}{\sqrt{\varepsilon}}\right]\text{   for some constants   }\beta_1,\,\,\beta_2>0.$$

Define $\omega^*_{x_i,\mu_i}$ and $\omega^*_{x,\mu}$ as follows.
\begin{equation}
\omega^*_{x_i,\mu_i}(y)
=\left\{\begin{split}
&V_{x_i,\mu_i}(y)+8\pi\gamma(y,x_i)(1-\frac{1}{\theta_i}),\,\,y\in B_{d_i}(x_i),\\
&V_{0,\mu_i}(d_i)+8\pi(G(y,x_i)-\frac{1}{2\pi}\ln\frac{1}{d_i})(1-\frac{1}{\theta_i}),\,\,y\in \Omega\setminus B_{d}(x_i),
\end{split}
\right.
\end{equation}
where $\frac{1}{\theta_i}=\frac{1}{1+(\mu_id_i)^2}$ which makes $\omega^*_{x_i,\mu_i}\in C^1(\Omega)$, and
$$\omega^*_{x,\mu}=\sum_{i=1}^k \omega_{x_i,\mu_i}^*.$$
Set $\omega_{\mu}$
\begin{equation}\label{eq03}
\omega_{\mu}(y)=  \omega_{x ,\mu }^*(y)-\int_{\Omega }  \omega_{x ,\mu}^* .
\end{equation}

For $k\geq 2$, we construct the approximation solution $(U_{1,\mu},U_{2,\mu})$ as follows:

\begin{equation}\label{eq04}
(U_{1,\mu},U_{2,\mu})=(\omega_{\mu}+c_{1,\mu},\omega_{\mu}+c_{2,\mu}),
\end{equation}
where
$$c_{1,\mu}=\ln \frac{16 k\pi\varepsilon^2}{\int_{\Omega}e^{u_{0,1}+\omega_{\mu}}\left(1+\sqrt{1-32k\pi\varepsilon^2\frac{ \int_{\Omega}e^{\sum_{i=1}^2 (u_{0,i}+\omega_{\mu})}     }{\int_{\Omega}e^{u_{0,1}+\omega_{\mu}}\int_{\Omega}e^{u_{0,2}+\omega_{\mu}}}}\right)             }  $$
and
$$c_{2,\mu}=\ln \frac{16 k\pi\varepsilon^2}{\int_{\Omega}e^{u_{0,2}+\omega_{\mu}}\left(1+\sqrt{1-32k\pi\varepsilon^2\frac{ \int_{\Omega}e^{\sum_{i=1}^2 (u_{0,i}+\omega_{\mu})}     }{\int_{\Omega}e^{u_{0,1}+\omega_{\mu}}\int_{\Omega}e^{u_{0,2}+\omega_{\mu}}}}\right)             }.  $$

\begin{remark}

When $k=1$, the approximation solution for \eqref{e001} has a simpler form:
\begin{equation}\label{eq55}\begin{pmatrix}U_{1,\mu} \\ U_{2,\mu}\end{pmatrix}=  \begin{pmatrix}\omega^*_{x_1,\mu_1}-u_{0,1}(x_1) \\ \omega^*_{x_1,\mu_1}-u_{0,2}(x_1)\end{pmatrix}. \end{equation}

\end{remark}\par

In the construction of approximation,  the height $\mu_i$ of the bubble $\omega_{x_i,\mu_i}$ cannot be an independent variables.   So, we set
$$\mu_1=\mu,\,\,\mu_i=\sqrt{\frac{\rho_1}{\rho_i}} \mu,\,\,i=2,\cdots, k, \text{  for  some }\mu>0,    $$
where
$$\rho_i= e^{8\pi\gamma(x_i,x_i)   +8\pi\sum_{j\not =i}G(x_i,x_j)+u_{0,1}(x_i)},\,\,i=1,\cdots,k.$$
It is clear that
\begin{equation}\label{eq01}
\rho_1 \mu_1^2=\rho_i\mu_i^2,\,\,i=2,\cdots,k.
\end{equation}
Denote $\rho^*_i= e^{8\pi\gamma(x_i,x_i)   +8\pi\sum_{j\not =i}G(x_i,x_j)+u_{0,2}(x_i)},\,\,i=1,\cdots,k.$
By using $(u_{0,1}-u_{0,2}(x_i))-(u_{0,1}-u_{0,2}(q_i))=O(\frac{1}{\mu})$, we obatin
\begin{equation}\label{eq02}
\rho^*_1 \mu_1^2=\rho^*_i\mu_i^2\left(1+O(\frac{1}{\mu})\right),\,\,i=1,\cdots,k.
\end{equation}
In the rest of this section, we will some properties of the approximation solution $(U_{1,\mu}, U_{2,\mu})$.\par

For $l=1,2$, $j=1,\cdots,k$ and $y\in B_{d_j}(x_j)$, we have
\begin{equation}
\begin{split}
&u_{0,l}(y)-u_{0,l}(x_i)+8\pi\left((\gamma(y,x_i)-\gamma(x_i,x_i))(1-\frac{1}{\theta_i})+\sum_{m\neq j}(G(y,x_m)-G(x_j,x_m))(1-\frac{1}{\theta_m})\right)\\
=&\left<Du_{0,l}(x_i)+8\pi\sum_{m\not =j}DG(x_i,x_m),y-x_i \right>+O(|y-x_i|^2+\frac{1}{\mu^2}).
\end{split}
\end{equation}
By this and \eqref{eq62}, we find that
\begin{equation}\label{eq56}\begin{split}
&\int_{B_{d_i}(x_i)} e^{\omega_{\mu}^*+u_{0,1}}\\
=& \frac{8^{k-1}\rho_i\mu_i^2}{\mu_1^2\cdots\mu_k^2}\left(1+O(\frac{1}{\mu^2})\right)\\
&\times \int_{B_{d_i}(x_i)} e^{V_{x_i,\mu_i}+u_{0,i}(y)-u_{0,i}(x_j)+8\pi\left((\gamma(y,x_j)-\gamma(x_j,x_j)) +\sum_{m\not =j}(G(y,x_m)-G(x_j,x_m))\right)     }\\
=&\frac{8^{k-1}\rho_i\mu_i^2}{\mu_1^2\cdots\mu_k^2}\left(8\pi+\int_{B_{d_i}(x_i)}e^{V_{x_i,u_i}}|y-x_i|^2+O(\frac{1}{\mu^2})\right)=\frac{8^{k-1}\rho_i\mu_i^2}{\mu_1^2\cdots\mu_k^2}\left(8\pi+O(\frac{\ln\mu}{\mu^2})\right).
\end{split}
\end{equation}
Similarly,
\begin{equation}\label{eq57}
 \int_{B_{d_i}(x_i)} e^{\omega_{\mu}^*+u_{0,2}} =\frac{8^{k-1}\rho^*_i\mu_i^2}{\mu_1^2\cdots\mu_k^2}\left(8 \pi+O(\frac{\ln\mu}{\mu^2})\right).
 \end{equation}
By \eqref{eq01}, \eqref{eq02}, \eqref{eq56} and \eqref{eq57}, we obtain
\begin{equation}\label{a04}
\int_{\Omega} e^{\omega_{\mu}^*+u_{0,1}}= \sum_{i=1}^k \frac{8^{k-1}\rho_i\mu_i^2}{\mu_1^2\cdots\mu_k^2}\left(8\pi+O(\frac{\ln\mu}{\mu^2})\right)=\frac{8^{k-1}\rho_i\mu_i^2}{\mu_1^2\cdots\mu_k^2}\left(8k\pi+O(\frac{\ln\mu}{\mu^2})\right).
\end{equation}
Similarly,
\begin{equation}\label{a05}
\int_{\Omega}e^{\omega_{\mu}^*+u_{0,2}}=\frac{8^{k-1}\rho^*_i\mu_i^2}{\mu_1^2\cdots \mu_k^2}\left(8k\pi+O(\frac{1}{\mu})\right),
\end{equation}
where \eqref{eq02} is used.
The asymptotic behaviours of $\omega_{\mu }$ and $c_{i,\mu}$(i=1,2) will be discussed in the following proposition.
\begin{proposition}\label{P1}
(1) The function $\omega_{\mu }$  defined in \eqref{eq03} satisfies that, for $\delta>0$,
\begin{equation}\left\{
\begin{split}
&\max_{y\in B_{\delta}(x_i)}\omega_{\mu}(y)=2\ln\frac{1}{\varepsilon}+O(1),\\
&\omega_{\mu}(y)=O(1),\,y\in \Omega\setminus\cup_{i=1}^k B_{\delta}(x_i).
\end{split}\right.
\end{equation}
(2) The constants $c_{1,\mu}$ and $c_{2,\mu}$ satisfy
\begin{equation}
c_{i,\mu}=-3\ln \frac{1}{\varepsilon}+O(1),\,\,\,i=1,2.
\end{equation}

\end{proposition}
\begin{proof}
The estimate of $\omega_{\mu}$ is exactly the same as Proposition 2.1 in \cite{LY2}. So, we omit the proof. \par
By definition of $c_{1,\mu}$, we deduce that, for $\mu>0$ large,
\begin{equation}\label{eq21}
e^{c_{1,\mu}}=\frac{16k\pi \varepsilon^2}{ \int_{\Omega}e^{\omega_{\mu}+u_{0,1}}\left(2-16k\pi\varepsilon^2 \frac{ \int_{\Omega}e^{\sum_{i=1}^2 (\omega_{\mu}+u_{0,i})}}{\int_{\Omega}e^{\omega_{\mu}+u_{0,1}}\int_{\Omega}e^{ \omega_{\mu}+u_{0,2}}}+O((\varepsilon^2\frac{ \int_{\Omega}e^{\sum_{i=1}^2 (\omega_{\mu}+u_{0,i})}}{\int_{\Omega}e^{\omega_{\mu}+u_{0,1}}\int_{\Omega}e^{ \omega_{\mu}+u_{0,2}}})^2)   \right)   }
\end{equation}
On the other hand,
\begin{equation}\label{eq22}
\frac{ \int_{\Omega}e^{\sum_{i=1}^2 (\omega_{\mu}+u_{0,i})}}{\int_{\Omega}e^{\omega_{\mu}+u_{0,1}}\int_{\Omega}e^{ \omega_{\mu}+u_{0,2}}}=\frac{ \int_{\Omega}e^{\sum_{i=1}^2 (\omega_{\mu}^*+u_{0,i})}}{\int_{\Omega}e^{\omega^*_{\mu}+u_{0,1}}\int_{\Omega}e^{ \omega^*_{\mu}+u_{0,2}}}=O(\mu^2).
\end{equation}
So, by \eqref{eq21} and \eqref{eq22},

\begin{equation}
e^{c_{1,\mu}}=\frac{8k\pi \varepsilon^2}{ \int_{\Omega}e^{\omega_{\mu}+u_{0,1}}}\left(1+8k\pi\varepsilon^2 \frac{ \int_{\Omega}e^{\sum_{i=1}^2 (\omega_{\mu}+u_{0,i})}}{\int_{\Omega}e^{\omega_{\mu}+u_{0,1}}\int_{\Omega}e^{ \omega_{\mu}+u_{0,2}}}+O(\varepsilon^4\mu^4)  \right),
\end{equation}
and thus
\begin{equation}
\begin{split}
c_{1,\mu}=&\ln(8k\pi)+2\ln \varepsilon-\ln\int_{\Omega}e^{\omega_{\mu}+u_{0,1}}+8k\pi\varepsilon^2  \frac{ \int_{\Omega}e^{\sum_{i=1}^2 (\omega_{\mu}+u_{0,i})}}{\int_{\Omega}e^{\omega_{\mu}+u_{0,1}}\int_{\Omega}e^{ \omega_{\mu}+u_{0,2}}}+O(\varepsilon^4\mu^4)\\
=&2\ln\varepsilon-2\ln\mu+O(1)\\
=&-3\ln\frac{1}{\varepsilon}+O(1).
\end{split}
\end{equation}
Similarly, $c_{2,\mu}=-3\ln\frac{1}{\varepsilon}+O(1) $
\end{proof}

In the  end of this section we estimate
$$-\Delta U_{1,\mu}+\frac{1}{\varepsilon^2} e^{U_{2,\mu}+u_{0,2}}(e^{U_{1,\mu}+u_{0,1}}-1)+8k\pi$$
and
$$
-\Delta U_{2,\mu}+\frac{1}{\varepsilon^2} e^{U_{1,\mu}+u_{0,1}}(e^{U_{2,\mu}+u_{0,2}}-1)+8k\pi$$
which will be useful in the next section.

\begin{proposition}
Let $(U_{1,\mu},U_{2,\mu})$ be defined  in \eqref{eq04}. Then

\begin{equation}\label{eq05}
\begin{split}
&-\Delta U_{1,\mu}+\frac{1}{\varepsilon^2} e^{U_{2,\mu}+u_{0,2}}(e^{U_{1,\mu}+u_{0,1}}-1)+8k\pi\\
&=\sum_{i=1}^{k}\left(1_{B_{d_i(x_i)}}e^{V_{x_i,\mu_i}}(1-e^{f_{2,i,{\bf x},\mu}})+\frac{8\pi}{\theta_i}\right)\\
&+O\left(  \frac{\ln \mu}{\mu^2} +  \sum_{i=1}^k( \frac{1}{\mu^2}e^{V_{x_i,\mu_i}+u_{0,2}}+\frac{1}{\mu^4}e^{2V_{x_i,\mu_i}+u_{0,1}+u_{0,2}}   ) \right)
\end{split}
\end{equation}

and
\begin{equation}\label{eq06}
\begin{split}
&-\Delta U_{2,\mu}+\frac{1}{\varepsilon^2} e^{U_{1,\mu}+u_{0,1}}(e^{U_{2,\mu}+u_{0,2}}-1)+8k\pi\\
&=\sum_{i=1}^{k}\left(1_{B_{d_i(x_i)}}e^{V_{x_i,\mu_i}}(1-e^{f_{1,i,{\bf x},\mu}})+\frac{8\pi}{\theta_i}\right)\\
&+O\left(  \frac{\ln \mu}{\mu^2} + \sum_{i=1}^k( \frac{1}{\mu^2}e^{V_{x_i,\mu_i}+u_{0,1}}+\frac{1}{\mu^4}e^{2V_{x_i,\mu_i}+u_{0,1}+u_{0,2}}   ) \right)                      .
\end{split}
\end{equation}
where
\begin{equation}\label{a09}\begin{split}
 f_{1,i,{\bf x},\mu}(y)=&\textit{•}8\pi\left((\gamma(y,x_i)-\gamma(x_i,x_i))(1-\frac{1}{\theta_i})+\sum_{j\not= i}(G(y,x_i)-G(x_i,x_j))(1-\frac{1}{\theta_j})\right)\\
&+u_{0,1}(y)-u_{0,1}(x_i)
\end{split}
\end{equation}
\begin{equation}\label{a10}\begin{split}
 f_{2,i,{\bf x},\mu}(y)=&\textit{•}8\pi\left((\gamma(y,x_i)-\gamma(x_i,x_i))(1-\frac{1}{\theta_i})+\sum_{j\not= i}(G(y,x_i)-G(x_i,x_j))(1-\frac{1}{\theta_j})\right)\\
&+u_{0,2}(y)-u_{0,2}(x_i)
\end{split}
\end{equation}

\end{proposition}

\begin{proof}
We only prove \eqref{eq05} since the proof of \eqref{eq06} is similar.\\
By the definition of $U_{1,\mu}$, we find that
\begin{equation}\label{a06}
\begin{split}
-\Delta U_{1,\mu}
&=-\Delta\omega_{\mu}\\
&=\sum_{i=1}^k\left(1_{B_{d_i(x_i)}}e^{V_{x_i,\mu_i}}-8\pi(1-\frac{1}{\theta_i})\right)
\end{split}
\end{equation}
where $1_A=1$ in A and $1_A=0$ otherwise.\\
Using \eqref{a04} and \eqref{a05}, we have
\begin{equation}\label{a07}
\begin{split}
&\frac{1}{\varepsilon^2}e^{U_{2,\mu}+u_{0,2}}\\
=&\frac{8k\pi e^{\omega_{\mu}+u_{0,2}}}{\int_{\Omega} e^{\omega_{\mu}+u_{0,2}}}\left(1+8k\pi\varepsilon^2\frac{\int_{\Omega}e^{\sum_{i=1}^2(\omega_{\mu}+u_{0,i})}}{\int_{\Omega} e^{\omega_{\mu}+u_{0,1}} \int_{\Omega} e^{\omega_{\mu}+u{0,2}}}+O(\varepsilon^4\mu^4)\right)
\end{split}
\end{equation}
and
\begin{equation}\label{a08}
\begin{split}
&\frac{1}{\varepsilon^2}e^{\sum_{i=1}^2(U_{i,\mu}+u_{0,i})}\\
=&64k^2\pi^2\varepsilon^2\left(
\frac{e^{\omega_{\mu}+u_{0,1}}e^{\omega_{\mu}+u_{0,2}}}{\int_{\Omega}e^{\omega_{\mu}+u_{0,1}}\int_{\Omega}e^{\omega_{\mu}+u_{0,2}}}\right)\left(1+16k\pi\varepsilon^2\frac{\int_{\Omega}e^{\sum_{i=1}^2(\omega_{\mu}+u_{0,i})}}{\int_{\Omega}e^{\omega_{\mu}+u_{0,1}} \int_{\Omega}e^{\omega_{\mu}+u_{0,2}}}+O(\varepsilon^4\mu^4)\right).
\end{split}
\end{equation}
Combining \eqref{a06}, \eqref{a07} and \eqref{a08}, we obtain
\begin{equation}\label{a11}
\begin{split}
&-\Delta U_{1,\mu}+\frac{1}{\varepsilon^2}e^{U_{2,\mu}+u_{0,2}}(e^{U_{1,\mu}+u_{0,1}}-1)+8k\pi\\
&=\sum_{i=1}^k \left(1_{B_{d_i(x_i)}}e^{V_{x_i,\mu_i}}+\frac{8\pi}{\theta_i}\right)-\frac{8k\pi e^{\omega_{\mu}+u_{0,2}}}{\int_{\Omega}e^{\omega_{\mu}+u_{0,2}}}+K_{1,u}+R_{1,\mu}
\end{split}
\end{equation}
where
$$K_{1,\mu}=\frac{64k^2\pi^2\varepsilon^2  \int_{\Omega}e^{  2\omega_{\mu}+u_{0,1}+u_{0,2}}}{\int_{\Omega}e^{\omega_{\mu}+u_{0,1}}\int_{\Omega}e^{ \omega_{\mu}+u_{0,2}}} \left[ \frac{e^{2\omega_{\mu}+u_{0,1}+u_{0,2}}}{\int_{\Omega}e^{2\omega_{\mu}+u_{0,1}+u_{0,2}}}-\frac{e^{ \omega_{\mu} +u_{0,2}}}{\int_{\Omega}e^{\omega_{\mu} +u_{0,2}}}                           \right] $$
and
$$
R_{1,\mu}=O\left( \frac{\mu^4\varepsilon^4 e^{\omega_{\mu}+u_{0,2} }}{\int_{\Omega}  e^{\omega_{\mu}+u_{0,2} }}
+ \frac{ \varepsilon^4 e^{2\omega_{\mu}+u_{0,1}+u_{0,2}}  \int_{\Omega}e^{  2\omega_{\mu}+u_{0,1}+u_{0,2}}}{(\int_{\Omega}e^{\omega_{\mu}+u_{0,1}})^2(\int_{\Omega}e^{ \omega_{\mu}+u_{0,2}})^2}+\frac{ \varepsilon^6\mu^4   e^{  2\omega_{\mu}+u_{0,1}+u_{0,2}}}{\int_{\Omega}e^{\omega_{\mu}+u_{0,1}}\int_{\Omega}e^{ \omega_{\mu}+u_{0,2}}}   \right).
$$
Note that
$$\int_{\Omega}e^{\sum_{i=1}^2(\omega_{\mu}^*+u_{0,i})}=O\left(\frac{\mu^2}{\mu^{4(k-1)}}\right)$$
By this, and \eqref{a04} and \eqref{a05}, we obtain
$$K_{1,\mu}=O\left(  \frac{1}{\mu^2}e^{V_{x_i,\mu_i}+u_{0,2}}+\frac{1}{\mu^4}e^{2V_{x_i,\mu_i}+u_{0,1}+u_{0,2}}                           \right)   $$
and
$$R_{1,\mu}=O\left(\frac{1}{\mu^4}e^{V_{x_i,\mu_i}+u_{0,2}}+\frac{1}{\mu^6}e^{2V_{x_i,\mu_i}+u_{0,1}+u_{0,2}}                   \right)$$
Next, by \eqref{a05}
\begin{equation}\label{a12}
\frac{8k\pi e^{\omega_{\mu}+u_{0,2}}}{\int_{\Omega}e^{\omega_{\mu}+u_{0,2}}}
= \frac{8k\pi e^{\omega_{\mu}^*+u_{0,2}}}{\int_{\Omega}e^{\omega_{\mu}^*+u_{0,2}}}
= \sum_{i=1}^k 1_{B_{d_i(x_i)}}e^{V_{x_i,\mu_i}+f_{1,i,\mu}}+O(\frac{1}{\mu^2})
\end{equation}
By this, we are led to \eqref{eq05}.
\end{proof}

\section{The Reduction and The Existence}

In this section, we will use the contraction mapping theorem to show that there exist ${\bf x}=(x_1,\cdots, x_k)$ and ${\bf \mu}=(\mu_1,\cdots,\mu_k)$ such that \eqref{e001} has a solution of this form $
\begin{pmatrix}
u_{1}\\
u_{2}
\end{pmatrix}
=
\begin{pmatrix}
U_{1,\mu}\\
U_{2,\mu}
\end{pmatrix}
+
\begin{pmatrix}
\omega_{1,\mu}\\
\omega_{2,\mu}
\end{pmatrix}
$
 where $\begin{pmatrix}
\omega_{1,\mu}\\
\omega_{2,\mu}
\end{pmatrix}$ is a perturbation term. For this purpose, we will use the linear operator $L_{\mu}$(see \eqref{eq23} below). However, the linear operator $L_{\mu}$ has non-trivial kernel, hence, we only can use the contraction mapping theorem to solve \eqref{e001} up to its kernel. Then we use the non-degeneracy  of $G_1^*+G_2^*$ at the blow-up point $\bf{q}$
and $\mathcal{D}^{(2)}(\bf{q})<0$ to find suitable $\bf{x}$ and $\bf{\mu}$ such that \eqref{e001} has a real solution whose blow-up set is $\{q_1,\cdots, q_k\}$.   Here, we only present the multi-bubble case($k\geq 2$). For the case $k=1$, it can be easily deduced form the multi-bubble case by using the approximation solution \eqref{eq55}. \par

  In view of $U_{1,\mu},U_{2,\mu}$, we consider the following simplified operator
\begin{equation}\label{eq23}
L_{\mu}\begin{pmatrix}
v_{1}\\
v_{2}
\end{pmatrix}=
\begin{pmatrix}
0&\sum_{i=1}^k1_{B_{d_i}(x_i)} e^{V_{x_i,\mu_i}}\\
\sum_{i=1}^k1_{B_{d_i}(x_i)} e^{V_{x_i,\mu_i}}&0
\end{pmatrix}
\begin{pmatrix}
v_{1}\\
v_{2}
\end{pmatrix}
\end{equation}
Then,  $(\omega_{1,\mu},\omega_{2,\mu})$ satisfies
\begin{equation}
L_{\mu}\begin{pmatrix}
\omega_{1,\mu}\\
\omega_{2,\mu}
\end{pmatrix}=
\begin{pmatrix}
 g_{1,\mu}\\
 g_{2,\mu}
\end{pmatrix},
\end{equation}
where
$$
g_{1,\mu}(x,t_1,t_2)=\sum_{i=1}^k1_{B_{d_i}(x_i)} e^{V_{x_i,\mu_i}} t_2-\frac{1}{\varepsilon}e^{U_{2,\mu}+u_{0,2}+t_2}(1-e^{U_{1,\mu}+u_{0,1}+t_1})-\Delta U_{1,\mu}+8k\pi,
$$
and
$$
g_{2,\mu}(x,t_1,t_2)=\sum_{i=1}^k1_{B_{d_i}(x_i)} e^{V_{x_i,\mu_i}} t_1-\frac{1}{\varepsilon}e^{U_{1,\mu}+u_{0,1}+t_1}(1-e^{U_{2,\mu}+u_{0,2}+t_2})-\Delta U_{2,\mu}+8k\pi.
$$
To apply the contraction argument, we first introduce  two function spaces:
$$X_{\alpha,\mu,2}\,\,\text{   and   }\,\,Y_{\alpha,\mu,2}.$$
Fix  a small fixed constant $\alpha>0$, we define
$$\rho(x)=(1+|x|)^{1+\frac{\alpha}{2}},\quad \hat{\rho}(x)=\frac{1}{(1+|x|)(\log (2+|x|))^{1+\frac{\alpha}{2}}}.$$
Denote  $\Omega'= \cup_{j=1}^k B_{d_j}(x_{j})$.
We say $\begin{pmatrix}
\xi_1\\
\xi_2
\end{pmatrix}
$
is in $X_{\alpha,\mu,2 }$
if
\begin{align}
\left\Vert \begin{pmatrix}
\xi_1\\
\xi_2
\end{pmatrix}\right\Vert^2_{X_{\alpha,\mu,2}}&=\sum_{j=1}^k \sum_{i=1}^2\left(||\Delta\widetilde{\xi}_{i,j}\rho||^2_{L^2(B_{  2d_j \mu_j} )}+||\widetilde{\xi}_{i,j}\hat{\rho}||^2_{L^2(B_{  2d_j \mu_j} )} \right)\\
&+  \sum_{i=1}^2 \left(|| \Delta\xi_i||^2_{L^2(  \Omega')} +||\xi_i||_{L^2(  \Omega')}^2 \right)<+\infty,
\end{align}
where $\widetilde{\xi}_{i,j}(y)=\xi(q_j+\frac{1}{\mu_j} y)$ and $B_d=B_d(0)$;
$\begin{pmatrix}
\xi_1\\
\xi_2
\end{pmatrix}
$
is in $Y_{\alpha,\mu,2}$
if
\begin{equation}
\left\Vert \begin{pmatrix}
\xi_1\\
\xi_2
\end{pmatrix}\right\Vert_{Y_{\alpha,\mu }}= \sum_{j=1}^N\sum_{i=1}^2\left(\frac{1}{\mu_j^4}||\widetilde{\xi}_{i,j}\rho||^2_{L^2( B_{2d_j\mu_j})}\right)+\sum_{i=1}^2\left(||\xi_i||^2_{L^2(\Omega\setminus \Omega' )}\right)<+\infty.
\end{equation}
We also denote $||\xi ||_{X_{\alpha,\mu}}= \left\Vert \begin{pmatrix}
\xi\\
0
\end{pmatrix}\right\Vert_{X_{\alpha,\mu,2}}$ and  $||\xi ||_{Y_{\alpha,\mu}}=  \left\Vert\begin{pmatrix}
\xi\\
0
\end{pmatrix}\right\Vert_{Y_{\alpha,\mu,2}}$.  So, we say $\xi\in X_{\alpha,\mu}$ if $||\xi||_{X_{\alpha,\mu}}<+\infty$ and
 $\xi\in Y_{\alpha,\mu}$ if $||\xi||_{Y_{\alpha,\mu}}<+\infty$.\par

Consider the cut-off function $\chi_j \in C^{\infty}(\mathbb{R}^2)$ satisfying
$$
\chi_j(x)=\left\{
\begin{array}{ccc}
1&\text{ for }& |x|\leq d_j\\
0&\text{ for }& |x|\geq 2d_j
\end{array}\right.
$$
and $0\leq \chi_j\leq 1$.
Next, we define the approximated kernel for $L_{\mu}$ as follows:

\begin{equation}
\begin{pmatrix}
Y_{1,0}\\
Y_{2,0}
\end{pmatrix}=
\begin{pmatrix}
-\frac{1}{\mu_1}+\sum_{i=1}^k \sqrt{\frac{\rho_1}{\rho_i}} \frac{2\chi_i(y-x_i)}{\mu_i(1+\mu_i|y-x_i|^2)}   \\
-\frac{1}{\mu_1}+\sum_{i=1}^k \sqrt{\frac{\rho_1}{\rho_i}} \frac{2\chi_i(y-x_i)}{\mu_i(1+\mu_i|y-x_i|^2)} ,
\end{pmatrix}
\end{equation}
\begin{equation}
\begin{pmatrix}
Y_{1,i,j}\\
Y_{2,i,j}
\end{pmatrix}=
\chi_{j}(|x-p_{j,\varepsilon}|)
\begin{pmatrix}
 \frac{\mu_j^2(y_i-x_{j,i})}{1+\mu_j|y-x_i|^2} \\
   \frac{\mu_j^2(y_i-x_{j,i})}{1+\mu_j|y-x_i|^2}
\end{pmatrix},\quad i=1,2,\quad j=1,\cdots, k.
\end{equation}
After calculation, it is not difficult to see that

$$L_{\mu}\begin{pmatrix}
Y_{1,0}\\
Y_{2,0}
\end{pmatrix}=O\left(\frac{1}{\mu^3}\right) ,\,\,L_{\mu}\begin{pmatrix}
Y_{1,i,j}\\
Y_{2,i,j}
\end{pmatrix}=O(1),\,\,i=1,2,\,\,j=1,\cdots,k.$$

Let

\begin{equation}
\begin{pmatrix}
Z_{1,0}\\
Z_{2,0}
\end{pmatrix}= \begin{pmatrix}
\Delta Y_{1,0}\\
\Delta Y_{2,0}
\end{pmatrix}
\end{equation}
and
\begin{equation}
\begin{pmatrix}
Z_{1,i,j}\\
Z_{2,i,j}
\end{pmatrix}=\begin{pmatrix}
\Delta Y_{1,i,j}\\
\Delta Y_{2,i,j}
\end{pmatrix},\quad i=1,2,\quad j=1,\cdots, k.
\end{equation}
Note that  $\begin{pmatrix}
Y_{1,0}\\
Y_{2,0}
\end{pmatrix}$, $\begin{pmatrix}
Y_{1,i,j}\\
Y_{2,i,j}
\end{pmatrix}$, $\begin{pmatrix}
Z_{1,0}\\
Z_{2,0}
\end{pmatrix}$ and  $\begin{pmatrix}
Z_{1,i,j}\\
Z_{2,i,j}
\end{pmatrix}$
are doubly periodic. Let

$$E_{\mu,2}=\left\{\begin{pmatrix}
\omega_1\\
\omega_2
\end{pmatrix}\in X_{\alpha,\mu,2}: \int_{\Omega}\left< \begin{pmatrix}
\omega_1\\
\omega_2
\end{pmatrix} ,\begin{pmatrix}
Z_{1,0}\\
Z_{2,0}
\end{pmatrix}  \right>= \int_{\Omega}\left< \begin{pmatrix}
\omega_1\\
\omega_2
\end{pmatrix} ,\begin{pmatrix}
Z_{1,i,j}\\
Z_{2,i,j}
\end{pmatrix}  \right> =0,\, i=1,2,\, j=1,\cdots, k.    \right\}
$$

and
$$F_{\mu,2}=\left\{\begin{pmatrix}
\omega_1\\
\omega_2
\end{pmatrix}\in Y_{\alpha,\mu,2}  : \int_{\Omega}\left< \begin{pmatrix}
\omega_1\\
\omega_2
\end{pmatrix} ,\begin{pmatrix}
Y_{1,0}\\
Y_{2,0}
\end{pmatrix}  \right>= \int_{\Omega}\left< \begin{pmatrix}
\omega_1\\
\omega_2
\end{pmatrix} ,\begin{pmatrix}
Y_{1,i,j}\\
Y_{2,i,j}
\end{pmatrix}  \right> =0,\, i=1,2,\, j=1,\cdots, k.    \right\}.
$$

We define the project operator $Q_{\mu}:Y_{\alpha,\mu,2}\to F_{\mu,2}$ as follows:
\begin{equation}\label{e023}
Q_{\mu}\begin{pmatrix}
u_1\\
u_2
\end{pmatrix}=\begin{pmatrix}
u_1\\
u_2
\end{pmatrix}-c_0 \begin{pmatrix}
Z_{1,0}\\
Z_{2,0}
\end{pmatrix} -\sum_{i=1}^2\sum_{j=1}^k c_{i,j}\begin{pmatrix}
Z_{1,i,j}\\
Z_{2,i,j}
\end{pmatrix}
\end{equation}
where $c_0$, $c_{i,j}$ are chosen so that
\begin{equation}
Q_{\mu}\begin{pmatrix}
u_1\\
u_2
\end{pmatrix}\in F_{\mu,2}.
\end{equation}

It is standard to prove the following result( or see Lemma 3.1 in \cite{LY2}).
\begin{lemma}\label{lemma31}
There is a constant $C>0$, independent of ${\bf x}$ and $\mu$, such that

\begin{equation}
\left\Vert Q_{\mu}\begin{pmatrix}
u_1\\
u_2
\end{pmatrix}\right\Vert_{Y_{\alpha,\mu,2}}\leq C\left\Vert  \begin{pmatrix}
u_1\\
u_2
\end{pmatrix}\right\Vert_{Y_{\alpha,\mu,2}}.
\end{equation}
\end{lemma}

By using the contraction mapping argument, we obtain the following theorem.

\begin{theorem}Assume (A.1)-(A.4) in Theorem \ref{main-thm} hold.
There exists $\varepsilon_0>0$ such that for $\varepsilon\in(0,\varepsilon_0)$,  $|{\bf x}-{\bf q}|<\frac{C_1}{\mu}$
, and $\mu \in   (\frac{\beta_1}{\sqrt{\varepsilon}},  \frac{\beta_2}{\sqrt{\varepsilon}})$ for some constants $C_1,\beta_1,\beta_2>0$, then there exists $\begin{pmatrix}
\omega_{1,\mu}\\
\omega_{2,\mu}
\end{pmatrix}\in E_{\mu,2}$  satisfying
\begin{equation}\label{eq09}
Q_{\mu}\left( L_{\mu}\begin{pmatrix}\omega_{1,\mu}\\ \omega_{2,\mu}\end{pmatrix}-\begin{pmatrix}g_{1,\mu}\\ g_{2,\mu}\end{pmatrix}    \right)=\begin{pmatrix}0\\0  \end{pmatrix}.
\end{equation}

Furthermore, $\begin{pmatrix}
\omega_{1,\mu}\\
\omega_{2,\mu}
\end{pmatrix}$ is a $C^1$ map of $({\bf x}, \mu_1,\cdots,\mu_k)$ to $X_{\alpha,\mu,2}$ and
\begin{equation}
\left\Vert Q_{\mu}\begin{pmatrix}
\omega_{1,\mu}\\
 \omega_{2,\mu}
\end{pmatrix}\right\Vert_{L^{\infty}(\Omega)}
+
\left\Vert Q_{\mu}\begin{pmatrix}
\omega_{1,\mu}\\
 \omega_{2,\mu}
\end{pmatrix}\right\Vert_{X_{\alpha,\mu,2}}\leq \frac{C\ln \mu}{\mu^{2-\frac{\alpha}{2}}},
\end{equation}
where the constant $\alpha$ is the same constant as in $X_{\alpha,\mu,2}$.

\end{theorem}

\begin{proof}
By the Theorem \ref{inv-thm}, the equation \eqref{eq09} can be rewritten as

\begin{equation}
\begin{pmatrix}\omega_{1,\mu}\\ \omega_{2,\mu}\end{pmatrix}= B_{\mu}\begin{pmatrix}\omega_{1,\mu}\\ \omega_{2,\mu}\end{pmatrix}
:=(Q_{\mu}L_{\mu})^{-1}Q_{\mu}  \begin{pmatrix}g_{1,\mu}\\ g_{2,\mu}\end{pmatrix}  .
\end{equation}

Define
$$
S_{\mu}=\left\{\begin{pmatrix}\omega_{1 }\\ \omega_{2} \end{pmatrix} : \begin{pmatrix}\omega_{1 }\\ \omega_{2}\end{pmatrix} \in E_{\mu},\,\, \left\Vert \begin{pmatrix}\omega_{1 }\\ \omega_{2}\end{pmatrix} \right\Vert_{L^{\infty}(\Omega)}+ \left\Vert \begin{pmatrix}\omega_{1 }\\ \omega_{2}\end{pmatrix} \right\Vert_{X_{\alpha,\mu,2}}\leq \frac{1}{\mu}  \right\}.
$$

  Firstly, we show that $B_{\mu}$ maps $S_{\mu}$ to $S_{\mu}$.
By Lemma \ref{lemma31} and \eqref{eq40}, we have

\begin{equation}\label{eq59}
\left\Vert B_{\mu} \begin{pmatrix}\omega_{1,\mu}\\ \omega_{2,\mu}\end{pmatrix} \right\Vert_{L^{\infty}(\Omega)}+ \left\Vert \begin{pmatrix}\omega_{1,\mu }\\ \omega_{2,\mu}\end{pmatrix} \right\Vert_{X_{\alpha,\mu,2}}\leq C \ln\mu \left\Vert Q_{\mu}\begin{pmatrix}g_{1,\mu }\\ g_{2,\mu}\end{pmatrix} \right\Vert_{Y_{\alpha, \mu,2}}\leq   C \ln\mu \left\Vert  \begin{pmatrix}g_{1,\mu }\\ g_{2,\mu}\end{pmatrix} \right\Vert_{Y_{\alpha, \mu,2}}.
\end{equation}

Thus we need to estimate $\left\Vert  \begin{pmatrix}g_{1,\mu }\\ g_{2,\mu}\end{pmatrix} \right\Vert_{Y_{\alpha, \mu,2}}$.

Note that
\begin{equation}\label{eq10}
\begin{split}
&e^{U_{2,\mu}+u_{0,2}+t_2   }(1-e^{U_{1,\mu}+u_{0,1}+t_1   })\\
=&e^{U_{2,\mu}+u_{0,2}   }(1-e^{U_{1,\mu}+u_{0,1}   })+e^{U_{2,\mu}+u_{0,2}   }(1-e^{U_{1,\mu}+u_{0,1}   }) t_2-e^{\sum_{i=1}^2U_{i,\mu}+u_{0,i}}t_1\\
&+O\left(   (e^{U_{1,\mu}+u_{0,1}}   +     e^{U_{2,\mu}+u_{0,2}}+e^{\sum_{i=1}^2U_{i,\mu}+u_{0,i}}) (t_1^2+t_2^2)  \right)
\end{split}
\end{equation}
Thus, by \eqref{eq10}, we obtain

\begin{equation}
\begin{split}
g_{1,\mu}=& h_{\mu} \omega_{2,\mu}-\frac{1}{\varepsilon^2} e^{U_{2,\mu}+u_{0,2}+\omega_{2,\mu}}(1- e^{U_{1,\mu}+u_{0,1}+\omega_{1,\mu}} )+8k\pi-\Delta U_{1,\mu}\\
=& (h_{\mu}-\frac{1}{\varepsilon^2} e^{U_{2,\mu}+u_{0,2}} )\omega_{2,\mu}+O\left(\frac{1}{\varepsilon^2} e^{\sum_{i=1}^2U_{i,\mu}+u_{0,i}}(|\omega_{1,\mu}| +|\omega_{2,\mu}| )  \right)\\
&+ O\left((\frac{1}{\varepsilon^2} e^{ U_{1,\mu}+u_{0,1}}+\frac{1}{\varepsilon^2} e^{ U_{2,\mu}+u_{0,2}})(|\omega_{1,\mu}|^2 +|\omega_{2,\mu}|^2 )  \right)\\
&+O\left( \sum_{i=1}^k e^{V_{x_i,\mu_i}}(D(u_{0,2}(x_i) +8\pi \sum_{j\not= i}G(x_j,x_i))|y-x_i|+|y-x_i|^2      \right)\\
&+O\left(\frac{\ln \mu}{\mu^2}+\sum_{i=1}^k (\frac{1}{\mu^2}e^{V_{x_i,\mu_i}+u_{0,2}}+\frac{1}{\mu^4}e^{2V_{x_i,\mu_i}+u_{0,1}+u_{0,2}}  )           \right).
\end{split}
\end{equation}
By this, we have

\begin{equation}\label{eq11}
\begin{split}
&\frac{1}{\mu^4}|| g_{1,\mu}(\mu^{-1}y+x_i, \omega_{1,\mu}(\mu^{-1}y+x_i),\omega_{2,\mu}(\mu^{-1}y+x_i)    )\rho(y)  ||^2_{ L^2( B_{2d_i})}\\
\leq & C(\frac{\ln^2\mu}{\mu^4}+\varepsilon^4\mu^4)\left\Vert  \begin{pmatrix}
\omega_{1,\mu}\\
 \omega_{2,\mu}
\end{pmatrix}\right\Vert_{L^{\infty}(\Omega)}+\frac{C|DG_2^*(\bf{x})|^2}{\mu^2}\\
\leq & \frac{C}{\mu^{4-\alpha}}.
\end{split}
\end{equation}
On the other hand, by the definition of $(U_{1,\mu}, U_{2,\mu})$
\begin{equation}\label{eq12}
||g_{1,\mu}||_{L^2(\Omega\setminus\Omega')}\leq \frac{C}{\mu^2}.
\end{equation}
Similarly, we have
\begin{equation}\label{eq13}
\frac{1}{\mu^4}|| g_{2,\mu}(\mu^{-1}y+x_i, \omega_{1,\mu}(\mu^{-1}y+x_i),\omega_{2,\mu}(\mu^{-1}y+x_i)    )\rho(y)  ||^2_{L^2(B_{2d_i})}\leq \frac{C}{\mu^{4-\alpha}}
\end{equation}
and
\begin{equation}\label{eq14}
||g_{2,\mu}||_{L^2(\Omega\setminus\Omega')}\leq \frac{C}{\mu^2}.
\end{equation}
By \eqref{eq11}, \eqref{eq12}, \eqref{eq13}   and \eqref{eq14},
\begin{equation}\label{eq15}
\left\Vert  \begin{pmatrix}
g_{1,\mu}\\
 g_{2,\mu}
\end{pmatrix}\right\Vert_{Y_{\alpha,\mu,2}}\leq \frac{C}{\mu^{2-\frac{\alpha}{2}}}
\end{equation}
Next, we show that $B_{\mu}$ is a contraction map.
For any $(\omega_{1,\mu},\omega_{2,\mu})$ and $(\widetilde{  \omega_{1,\mu}  }, \widetilde{\omega_{2,\mu}}  )\in S_{\mu}$, as the calculation above, we have
\begin{equation}\label{eq16}\begin{split}
&\left\Vert B_{\mu} \begin{pmatrix}
\omega_{1,\mu}\\
 \omega_{2,\mu}
\end{pmatrix}-   B_{\mu} \begin{pmatrix}
\widetilde{\omega_{1,\mu}}\\
\widetilde{ \omega_{2,\mu}}
\end{pmatrix}\right\Vert_{ L^{\infty}(\Omega)}
+\left\Vert B_{\mu} \begin{pmatrix}
\omega_{1,\mu}\\
 \omega_{2,\mu}
\end{pmatrix}-   B_{\mu} \begin{pmatrix}
\widetilde{\omega_{1,\mu}}\\
\widetilde{ \omega_{2,\mu}}
\end{pmatrix}\right\Vert_{X_{\alpha,\mu,2}} \\
\leq &C\ln \mu \left\Vert   \begin{pmatrix}
g_{1,\mu}(x,  {\omega_{1,\mu}},    { \omega_{2,\mu}})\\
 g_{2,\mu}(x,  {\omega_{1,\mu}},   { \omega_{2,\mu}})
\end{pmatrix}-    \begin{pmatrix}
g_{1,\mu}(x, \widetilde{\omega_{1,\mu}},   \widetilde{ \omega_{2,\mu}})\\
g_{1,\mu}(x,\widetilde{\omega_{1,\mu}}, \widetilde{ \omega_{2,\mu}})
\end{pmatrix}\right\Vert_{X_{\alpha,\mu,2}} .
\end{split}
\end{equation}

Combing \eqref{eq15} and \eqref{eq16}, we have proved that $B_{\mu}$ is a contraction map. Furthermore, by the contraction mapping theorem, we know that there exists $(\omega_{1,\mu},\omega_{2,\mu})\in S_{\mu}$ satisfying
\begin{equation}\label{eq26}
\left\Vert  \begin{pmatrix}
\omega_{1,\mu}\\
 \omega_{2,\mu}
\end{pmatrix}\right\Vert_{L^{\infty}(\Omega)}
+
\left\Vert \begin{pmatrix}
\omega_{1,\mu}\\
 \omega_{2,\mu}
\end{pmatrix}\right\Vert_{X_{\alpha,\mu,2}}\leq \left\Vert  \begin{pmatrix}
g_{1,\mu}\\
g_{2,\mu}
\end{pmatrix}\right\Vert_{Y_{\alpha,\mu,2}} \leq \frac{C\ln \mu}{\mu^{2-\frac{\alpha}{2}}}.
\end{equation}
\end{proof}

By the above theorem, we obtain that for $|\bf{x}-\bf{q}|<\frac{1}{\mu}$ and $\mu\in[\frac{\beta_0}{\sqrt{\varepsilon}},\frac{\beta_0}{\sqrt{\varepsilon}}]$, then there exists a doubly periodic function $(\omega_{1,\mu},\omega_{2,\mu})$,
 $c_0$ and $c_{i,j}$, $i=1,2,$ $j=1,\cdots, k$, such that
 \begin{equation}
 \begin{split}
 &\Delta
 \begin{pmatrix}
 U_{1,\mu}+\omega_{1,\mu}\\
  U_{2,\mu}+\omega_{2,\mu}
 \end{pmatrix}
+  \begin{pmatrix}
\frac{1}{\varepsilon^2} e^{U_{2,\mu}+u_{0,2}+\omega_{2,\mu}}(1-e^{U_{1,\mu}+u_{0,1}+\omega_{1,\mu}})-8k\pi\\
 \frac{1}{\varepsilon^2} e^{U_{1,\mu}+u_{0,1}+\omega_{1,\mu}}(1-e^{U_{2,\mu}+u_{0,2}+\omega_{2,\mu}})-8k\pi
 \end{pmatrix}\\
=& c_0 \begin{pmatrix}
Z_{1,0}\\
Z_{2,0}
\end{pmatrix} +\sum_{i=1}^2\sum_{j=1}^k \begin{pmatrix}
Z_{1,i,j}\\
Z_{2,i,j}
\end{pmatrix}.
 \end{split}
 \end{equation}
To obtain a true solution for \eqref{e001}, we need to choose suitable $\bf{x}$ and $\bf{\mu}$ such that
\begin{equation}
c_0=c_{i,j}=0,\, i=1,2,\, j=1,\cdots,k.
\end{equation}
We begin with the following simple observation.

\begin{lemma}\label{lemma33}
If

\begin{equation} \label{eq07}
\int_{\Omega}\left<
\begin{pmatrix}
\Delta ( U_{1,\mu}+\omega_{1,\mu}))+\frac{1}{\varepsilon^2}e^{U_{2,\mu}+\omega_{2,\mu}+u_{0,2}}(1-e^{U_{1,\mu}+\omega_{1,\mu}+u_{0,1}})-8k\pi\\
\Delta (U_{2,\mu}+\omega_{1,\mu}))+\frac{1}{\varepsilon^2}e^{U_{1,\mu}+\omega_{1,\mu}+u_{0,1}}(1-e^{U_{2,\mu}+\omega_{2,\mu}+u_{0,2}})-8k\pi
\end{pmatrix}
,\begin{pmatrix}
Y_{1,0}\\
Y_{2,0}
\end{pmatrix}
\right>=0
\end{equation}
and for  $i=1,2,\,$ $j=1,\cdots, k,$
\begin{equation} \label{eq08}
\int_{\Omega}\left<
\begin{pmatrix}
\Delta ( U_{1,\mu}+\omega_{1,\mu}))+\frac{1}{\varepsilon^2}e^{U_{2,\mu}+\omega_{2,\mu}+u_{0,2}}(1-e^{U_{1,\mu}+\omega_{1,\mu}+u_{0,1}})-8k\pi\\
\Delta (U_{2,\mu}+\omega_{1,\mu}))+\frac{1}{\varepsilon^2}e^{U_{1,\mu}+\omega_{1,\mu}+u_{0,1}}(1-e^{U_{2,\mu}+\omega_{2,\mu}+u_{0,2}})-8k\pi
\end{pmatrix}
,\begin{pmatrix}
Y_{1,i,j}\\
Y_{2,i,j}
\end{pmatrix}
\right>=0,
\end{equation}
then
$c_0=c_{i,j}=0$, $i=1,2,$ $j=1,\cdots, k.$
\end{lemma}

We calculate the left hand side of \eqref{eq07} and \eqref{eq08} in the following two lemmas.

\begin{lemma}\label{lemma34}
For $i=1,2$, $j=1,\cdots,k$, there exist  $A_{i,j}>0$, such that

%
%
%
%
%
%

\begin{equation}\label{eq25}
\begin{split}
&\int_{\Omega}\left<
\begin{pmatrix}
\Delta ( U_{1,\mu}+\omega_{1,\mu}))+\frac{1}{\varepsilon^2}e^{U_{2,\mu}+\omega_{2,\mu}+u_{0,2}}(1-e^{U_{1,\mu}+\omega_{1,\mu}+u_{0,1}})-8k\pi\\
\Delta (U_{2,\mu}+\omega_{1,\mu}))+\frac{1}{\varepsilon^2}e^{U_{1,\mu}+\omega_{1,\mu}+u_{0,1}}(1-e^{U_{2,\mu}+\omega_{2,\mu}+u_{0,2}})-8k\pi
\end{pmatrix}
,\begin{pmatrix}
Y_{1,i,j}\\
Y_{2,i,j}
\end{pmatrix}
\right>\\
=& A_{i,j}\left(  \frac{\partial G_1^*(\bf{x})}{\partial x_{i,j}}  +\frac{\partial G_2^*(\bf{x})}{\partial x_{i,j}}               \right)+O\left(\frac{\ln \mu}{\mu^{2-\frac{\alpha}{2}}}\right)
\end{split}
\end{equation}
 \end{lemma}

\begin{proof}
We mainly use \eqref{eq05}, \eqref{eq06} and \eqref{eq26}  to obtain the estimate \eqref{eq25}. We calculate the case $j=1$ only. The others are similar. \par
Step 1. We  first calculate this term
\begin{equation}
\int_{\Omega}\left<
\begin{pmatrix}
\Delta U_{1,\mu}+\frac{1}{\varepsilon^2}e^{U_{2,\mu}+u_{0,2}}-8k\pi\\
\Delta U_{2,\mu}+\frac{1}{\varepsilon^2}e^{U_{1,\mu}+u_{0,1}}-8k\pi
\end{pmatrix}
,\begin{pmatrix}
Y_{1,i,1}\\
Y_{2,i,1}
\end{pmatrix}
\right> .
\end{equation}
By the definition of $(U_{1,\mu},U_{2,\mu})$ and symmetry, we find that
\begin{equation}
\begin{split}
&\int_{\Omega} \left<  \begin{pmatrix}
\Delta U_{1,\mu}-8k\pi\\
\Delta U_{2,\mu}-8k\pi
\end{pmatrix}
\begin{pmatrix}
Y_{1,i,1}\\
Y_{2,i,1}
\end{pmatrix}
\right>  \\
=& \int_{\Omega} 1_{B_{d_1(x_1)}} e^{V_{x_1,\mu_1}}Y_{1,i,1}+  1_{B_{d_1}(x_1)} e^{V_{x_1,\mu_1}}Y_{2,i,1} + \sum_{i=1}^k\frac{8\pi}{\theta_i}\int_{\Omega}Y_{1,i,1}+Y_{2,i,1}=0
\end{split}
\end{equation}
Next,
\begin{equation}\begin{split}
&\int_{\Omega}\left<
\begin{pmatrix}
 \frac{1}{\varepsilon^2}e^{U_{2,\mu}+u_{0,2}}-8k\pi\\
 \frac{1}{\varepsilon^2}e^{U_{1,\mu}+u_{0,1}}-8k\pi
\end{pmatrix}
,\begin{pmatrix}
Y_{1,i,1}\\
Y_{2,i,1}
\end{pmatrix}
\right>\\
=&\frac{8k\pi\int_{\Omega}  e^{\omega_{\mu}+u_{0,1}}Y_{1,i,1}}{\int_{\Omega}  e^{\omega_{\mu}+u_{0,1}}}+
\frac{8k\pi\int_{\Omega}  e^{\omega+u_{0,2}}Y_{2,i,1}}{\int_{\Omega}  e^{\omega+u_{0,2}}}\\
=&\frac{   \mu_2^2\cdots \mu_k^2 \int_{B_{d_1}(x_1)}e^{\omega^*_{\mu}+u_{0,2}}Y_{1,i,1}  }{8^{k-1}(\rho_1^*+O(\frac{1}{\mu}))}+\frac{   \mu_2^2\cdots \mu_k^2 \int_{B_{d_1}(x_1)}e^{\omega^*_{\mu}+u_{0,1}}Y_{2,i,1}  }{8^{k-1}(\rho_1 +O(\frac{1}{\mu}))}+O(\frac{1}{\mu^2})\\
=& \int_{\mathbb{R}^2}\frac{8}{(1+|y|^2)^2} \frac{|y|^2}{1+|y|^2}dy        \left(   \frac{\partial G_2^*(\bf{x})}{\partial x_{i,2}} +\frac{\partial G_1^*(\bf{x})}{\partial x_{i,1}} \right)+O(\frac{\ln \mu}{\mu^2}),
\end{split}
\end{equation}
where $(A_2)$ is used. We thus  denote $A_{i,j}=\int_{\mathbb{R}^2}\frac{8}{(1+|y|^2)^2} \frac{|y|^2}{1+|y|^2} dy     $.\\
Step 2. Next, we estimate the remainder term. By mean value theorem, there exist $t_i$ between 0 and $\omega_{i,\mu}$($i=1,2$), such that
\begin{equation}\label{eq28}
\begin{split}
&\int_{\Omega}\left<
\begin{pmatrix}
\Delta  \omega_{1,\mu}  +\frac{1}{\varepsilon^2}(e^{U_{2,\mu}+\omega_{2,\mu}+u_{0,2}}- e^{U_{2,\mu}+u_{0,2}}) \\
\Delta  \omega_{2,\mu}  +\frac{1}{\varepsilon^2}(e^{U_{1,\mu}  +\omega_{1,\mu}+u_{0,1}}-e^{U_{1,\mu}+u_{0,1}})
\end{pmatrix}
,\begin{pmatrix}
Y_{1,i,1}\\
Y_{2,i,1}
\end{pmatrix}
\right>\\
=& \int_{\Omega}\left<
\begin{pmatrix}
\Delta  \omega_{1,\mu}  +\frac{1}{\varepsilon^2} e^{U_{2,\mu}+u_{0,2}+t_2}\omega_{2,\mu} \\
\Delta  \omega_{2,\mu}  +\frac{1}{\varepsilon^2} e^{U_{1,\mu}+u_{0,1}+t_1}\omega_{1,\mu}
\end{pmatrix}
,\begin{pmatrix}
Y_{1,i,1}\\
Y_{2,i,1}
\end{pmatrix}
\right>\\
=&\int_{\Omega}\left<
\begin{pmatrix}
\Delta  \omega_{1,\mu}  + \frac{8\pi e^{\omega_{\mu}^*+u_{0,2}}} {\int_{\Omega} e^{\omega_{\mu}^*+u_{0,2}}}\omega_{2,\mu} \\
\Delta  \omega_{2,\mu}  +\frac{8\pi e^{\omega_{\mu}^*+u_{0,1}}} {\int_{\Omega} e^{\omega_{\mu}^*+u_{0,1}}}\omega_{1,\mu}
\end{pmatrix}
,\begin{pmatrix}
Y_{1,i,1}\\
Y_{2,i,1}
\end{pmatrix}
\right>+O\left(\mu    \left\Vert  \begin{pmatrix}
\omega_{1,\mu}\\
 \omega_{2,\mu}
\end{pmatrix}\right\Vert_{L^{\infty}(\Omega)}^2 \right)\\
=& \int_{\Omega}\left<
\begin{pmatrix}
\Delta  \omega_{1,\mu}  +  \sum_{m=1}^k1_{B_{d_m}(x_m)} e^{V_{x_m,\mu_m}}\omega_{2,\mu} \\
\Delta  \omega_{2,\mu}  +  \sum_{m=1}^k1_{B_{d_m}(x_m)} e^{V_{x_m,\mu_m}}  \omega_{1,\mu}
\end{pmatrix}
,\begin{pmatrix}
Y_{1,i,1}\\
Y_{2,i,1}
\end{pmatrix}
\right>+O\left(\mu    \left\Vert  \begin{pmatrix}
\omega_{1,\mu}\\
 \omega_{2,\mu}
\end{pmatrix}\right\Vert_{L^{\infty}(\Omega)}^2 \right)\\
\end{split}
\end{equation}
Since $(\omega_{1,\mu},\omega_{2,\mu})$ and $(Y_{1,i,1},Y_{2,i,1})$ are doubly periodic,
\begin{equation}\label{eq29}
\begin{split}
 &\int_{\Omega}\left<
\begin{pmatrix}
\Delta  \omega_{1,\mu}  +  \sum_{m=1}^k1_{B_{d_m}(x_m)} e^{V_{x_m,\mu_m}}\omega_{2,\mu} \\
\Delta  \omega_{2,\mu}  +  \sum_{m=1}^k1_{B_{d_m}(x_m)} e^{V_{x_m,\mu_m}}  \omega_{1,\mu}
\end{pmatrix}
,\begin{pmatrix}
Y_{1,i,1}\\
Y_{2,i,1}
\end{pmatrix}
\right>\\
=& \int_{\Omega}\left<
\begin{pmatrix}
\Delta  Y_{1,i,1} +  \sum_{m=1}^k1_{B_{d_m}(x_m)} e^{V_{x_m,\mu_m}} Y_{2,i,1}\\
\Delta  Y_{2,i,1} +  \sum_{m=1}^k1_{B_{d_m}(x_m)} e^{V_{x_m,\mu_m}}   Y_{1,i,1}
\end{pmatrix}
,\begin{pmatrix}
\omega_{1,\mu} \\
 \omega_{2,\mu}
\end{pmatrix}
\right>+O\left(\left\Vert  \begin{pmatrix}
\omega_{1,\mu}\\
 \omega_{2,\mu}
\end{pmatrix}\right\Vert_{L^{\infty}(\Omega)} \right)\\
=&O\left(\left\Vert  \begin{pmatrix}
\omega_{1,\mu}\\
 \omega_{2,\mu}
\end{pmatrix}\right\Vert_{L^{\infty}(\Omega)} \right).
\end{split}
\end{equation}
Similarly, we have

\begin{equation}\label{eq30}
\int_{\Omega}\left<
\begin{pmatrix}
 \frac{1}{\varepsilon^2}e^{\sum_{j=1}^2(U_{j,\mu}+\omega_{j,\mu}+u_{0,j})}  \\
\frac{1}{\varepsilon^2}e^{\sum_{j=1}^2(U_{j,\mu}+\omega_{j,\mu}+u_{0,j})}
\end{pmatrix}
,\begin{pmatrix}
Y_{1,i,1}\\
Y_{2,i,1}
\end{pmatrix}
\right>=O\left( \varepsilon^2\mu^2+\varepsilon^2\mu^3  \left\Vert  \begin{pmatrix}
\omega_{1,\mu}\\
 \omega_{2,\mu}
\end{pmatrix}\right\Vert_{L^{\infty}(\Omega)}  \right).
\end{equation}

By \eqref{eq26}, \eqref{eq28}, \eqref{eq29} and \eqref{eq30}, we conclude that
\begin{equation}\label{eq27}
\int_{\Omega}\left<
\begin{pmatrix}
\Delta  \omega_{1,\mu}  +\frac{1}{\varepsilon^2}e^{U_{2,\mu}+\omega_{2,\mu}+u_{0,2}}(1-e^{U_{1,\mu}+\omega_{1,\mu}+u_{0,1}})-\frac{1}{\varepsilon^2}e^{U_{2,\mu}+u_{0,2}} \\
\Delta  \omega_{1,\mu}  +\frac{1}{\varepsilon^2}e^{U_{1,\mu}+\omega_{1,\mu}+u_{0,1}}(1-e^{U_{2,\mu}+\omega_{2,\mu}+u_{0,2}})-\frac{1}{\varepsilon^2}e^{U_{1,\mu}+u_{0,1}}
\end{pmatrix}
,\begin{pmatrix}
Y_{1,i,1}\\
Y_{2,i,1}
\end{pmatrix}
\right> =O\left(\frac{\ln \mu}{\mu^{2-\frac{\alpha}{2}}}\right).
\end{equation}

\end{proof}

\begin{lemma}\label{lemma35}  There exists $B>0$, such that
\begin{equation} \label{eq31}
\begin{split}
&\int_{\Omega}\left<
\begin{pmatrix}
\Delta ( U_{1,\mu}+\omega_{1,\mu}))+\frac{1}{\varepsilon^2}e^{U_{2,\mu}+\omega_{2,\mu}+u_{0,2}}(1-e^{U_{1,\mu}+\omega_{1,\mu}+u_{0,1}})-8k\pi\\
\Delta (U_{2,\mu}+\omega_{1,\mu}))+\frac{1}{\varepsilon^2}e^{U_{1,\mu}+\omega_{1,\mu}+u_{0,1}}(1-e^{U_{2,\mu}+\omega_{2,\mu}+u_{0,2}})-8k\pi
\end{pmatrix}
,\begin{pmatrix}
Y_{1,0}\\
Y_{2,0}
\end{pmatrix}
\right>\\
=& \frac{8}{\mu^3\rho_1}\left( \sum_{i=1}^k\rho_i(\int_{\Omega_i\setminus B_{\delta}(x_i)} \frac{ e^{f_{1,i,{\bf x}}}-1}{|y-x_i|^4}-\int_{\mathbb{R}^2\setminus\Omega_i}\frac{1}{|y-x_i|^4}         )\right)\\
&+\frac{8}{\mu^3\rho^*_1}\left( \sum_{i=1}^k\rho^*_i(\int_{\Omega_i\setminus B_{\delta}(x_i)} \frac{e^{f_{2,i,{\bf x}}}-1}{|y-x_i|^4}-\int_{\mathbb{R}^2\setminus\Omega_i}\frac{1}{|y-x_i|^4}         )\right)\\
& +B\varepsilon^2\mu+O(\frac{\delta^2}{\mu^3})+O(\frac{\ln\mu}{\mu^{4-\frac{\alpha}{2}}})
\end{split}
\end{equation}

\end{lemma}
\begin{proof}
As in the proof of Lemma \eqref{lemma34}, we firstly calculate

\begin{equation}\label{eq34}
\begin{split}
&\int_{\Omega}\left<
\begin{pmatrix}
\Delta ( U_{1,\mu}+\omega_{1,\mu}))+\frac{1}{\varepsilon^2}e^{U_{2,\mu}+\omega_{2,\mu}+u_{0,2}}(1-e^{U_{1,\mu}+\omega_{1,\mu}+u_{0,1}})-8k\pi\\
\Delta (U_{2,\mu}+\omega_{1,\mu}))+\frac{1}{\varepsilon^2}e^{U_{1,\mu}+\omega_{1,\mu}+u_{0,1}}(1-e^{U_{2,\mu}+\omega_{2,\mu}+u_{0,2}})-8k\pi
\end{pmatrix}
,\begin{pmatrix}
Y_{1,0}\\
Y_{2,0}
\end{pmatrix}
\right>\\
=& \sum_{i=1}^k \int_{B_{d_i(x_i)}} e^{V_{x_i,\mu_i}}(e^{f_{2,i,{\bf x},\mu}}-1)Y_{1,0}+   e^{V_{x_i,\mu_i}}(e^{f_{1,i,{\bf x},\mu}})Y_{2,0} \\
+&\int_{\Omega}\frac{8\pi}{\theta_i}(Y_{1,0}+Y_{2,0})+\frac{\int_{\Omega \setminus  \Omega^{'}}e^{\omega_{\mu}^*+u_{0,2}}Y_{1,0}}{\int_{\Omega}e^{\omega_{\mu}^*+u_{0,2}}}+\frac{\int_{\Omega \setminus  \Omega^{'}}e^{\omega_{\mu}^*+u_{0,1}}Y_{2,0}}{\int_{\Omega}e^{\omega_{\mu}^*+u_{0,1}}}\\
\end{split}
\end{equation}
For $i=1,\cdots,n$, $j=1,2$, we write
\begin{equation}
e^{f_{j,i,{\bf x},\mu}(y)}-1 =e^{f_{j,i,\textbf{x}}(y)}-1+e^{f_{j,i,{\bf x},\mu}(y)}-e^{f_{j,i,\textbf{x}}(y)}
\end{equation}
Then
\begin{equation}\label{a13}
\begin{split}
&\int_{B_{d_i(x_i)}}\left(e^{V_{x_i,\mu_i}}(e^{f_{2,i,{\bf x},\mu}(y)}-1)Y_{1,0}+e^{V_{x_i,\mu_i}}(e^{f_{1,i,{\bf x},\mu}(y)}-1)Y_{2,0}\right)\\
=&\int_{B_{d_i(x_i)}}\left(e^{V_{x_i,\mu_i}}(e^{f_{2,i,\textbf{x}}}-1)Y_{1,0}+e^{V_{x_i,\mu_i}}(e^{f_{1,i,\textbf{x}}}-1)Y_{2,0}\right)+O(\frac{1}{\mu^4})\\
\end{split}
\end{equation}
A straifhtforward calculation shows that
\begin{equation}\label{a14}
\sum_{i=1}^k \int_{\Omega} \frac{8\pi}{\theta_i} (Y_{1,0}+Y_{2,0}) = -\frac{16\pi}{\mu} \sum_{i=1}^k \frac{1}{\theta_i}+O(\frac{\ln\mu}{\mu^5}).
\end{equation}
and
\begin{equation}\label{a15}
\begin{split}
&\frac{\int_{\Omega\setminus\Omega^{'}} e^{\omega_{\mu}+u_{0,2}} Y_{1,0}}{\int_{\Omega} e^{\omega_{\mu}^*+u_{0,2}}}+\frac{\int_{\Omega\setminus\Omega^{'}} e^{\omega_{\mu}^*+u_{0,1}} Y_{2,0}}{\int_{\Omega} e^{\omega_{\mu}^*+u_{0,1}}}
\\
=&-\frac{8}{\rho_1^* \mu^2}\left(1+O(\frac{1}{\mu})\right) \int_{\Omega\setminus\Omega^{'}} e^{u_{0,2}+8\pi \sum_{i=1}^k G(y,x_i)} \left(-\frac{1}{\mu}+O(\frac{1}{\mu^3}) \right)\\
&-\frac{8}{\rho_1 \mu^2}\left(1+O(\frac{1}{\mu^2})\right) \int_{\Omega\setminus\Omega^{'}} e^{u_{0,1}+8\pi \sum_{i=1}^k G(y,x_i)} \left(-\frac{1}{\mu}+O(\frac{1}{\mu^3}) \right)\\
=&-\frac{8}{\rho_1^* \mu^3} \int_{\Omega\setminus\Omega^{'}} e^{u_{0,2}+8\pi \sum_{i=1}^k G(y,x_i)} -\frac{8}{\rho_1 \mu^3} \int_{\Omega\setminus\Omega^{'}} e^{u_{0,1}+8\pi \sum_{i=1}^k G(y,x_i)}+O(\frac{1}{\mu^4})\\
\end{split}
\end{equation}
Fix a small positive constant $\delta$ with $\delta\ll d_i,\ i=1,\cdots,k.$
Note that $\Delta e^{f_{i,j,{\bf x}}} = e^{f_{i,j,{\bf x}}}\vert Df_{i,j,{\bf x}} \vert,\ i=1,\cdots,k,\ j=1,2$. By the symmetry and Taylor expansion on $\left(e^{f_{i,j,{\bf x}}}-1\right)$, we obtain
\begin{equation}\label{a16}
\begin{split}
&\int_{B_{\delta}(x_i)} e^{V_{x_i,\mu_i}} (e^{f_{2,i,\textbf{x}}}-1) Y_{1,0}+\int_{B_{\delta}(x_i)} e^{V_{x_i,\mu_i}} (e^{f_{1,i,\textbf{x}}}-1) Y_{2,0}\\
=&O\left(\frac{1}{\mu} \int_{B_{\delta}(x_i)} e^{V_{x_i,\mu_i}} \Big( |Df_{2,i,\textbf{x}}(x_i) |^2 |y-x_i|^2 +  |Df_{1,i,\textbf{x}}(x_i) |^2 |y-x_i|^2 + |y-x_i|^4      \Big) \right)\\
=&O\left(\frac{\delta^2}{\mu^3} + \frac{\ln\mu}{\mu^5} \right)\\
\end{split}
\end{equation}
By this, we find that
\begin{equation}\label{a17}
\begin{split}
&\int_{B_{\delta}(x_i)} e^{V_{x_i,\mu_i}} (e^{f_{2,i,\textbf{x}}}-1) Y_{1,0}+\int_{B_{\delta}(x_i)} e^{V_{x_i,\mu_i}} (e^{f_{1,i,\textbf{x}}}-1) Y_{2,0}\\
=&\int_{B_{d_i}(x_i) \setminus B_{\delta}(x_i)} e^{V_{x_i,\mu_i}} (e^{f_{2,i,\textbf{x}}}-1) Y_{1,0}+\int_{B_{d_i}(x_i) \setminus B_{\delta}(x_i)} e^{V_{x_i,\mu_i}} (e^{f_{1,i,\textbf{x}}}-1) Y_{2,0}+O\left(\frac{\delta^2}{\mu^3} + \frac{ln\mu}{\mu^5} \right)\\
=&-\frac{8}{\mu^3} \int_{B_{d_i}(x_i) \setminus B_{\delta}(x_i)} \frac{e^{f_{2,i,\textbf{x}}}-1}{|y-x_i|^4} -\frac{8}{\mu^3} \int_{B_{d_i}(x_i) \setminus B_{\delta}(x_i)} \frac{e^{f_{1,i,\textbf{x}}}-1}{|y-x_i|^4} + O\left(\frac{\delta^2}{\mu^3} + \frac{ln\mu}{\mu^5} \right)\\
\end{split}
\end{equation}
For $i=1,\cdots,k,$ we have
\begin{equation}\label{a18}
e^{u_{0,1}(y) + 8\pi\sum_{i=1}^k G(y,x_i)} = \frac{\rho_i}{|y-x_i|^4} e^{f_{1,i,\textbf{x}}(y)}
\end{equation}
and
\begin{equation}\label{a19}
e^{u_{0,2}(y) + 8\pi\sum_{i=1}^k G(y,x_i)} = \frac{\rho_i^*}{|y-x_i|^4} e^{f_{2,i,\textbf{x}}(y)}
\end{equation}
Recall that $\frac{1}{\theta_i} = \frac{1}{1+(\mu_id_i)^2}$. So, we have
\begin{equation}\label{a20}
\frac{1}{\theta_i} = \int_{ \mathbb{R}^2\setminus B_{d_i}(x_i)} \frac{1}{|y-x_i|^4}dy + O(\frac{1}{\mu^2})
\end{equation}
Combining \eqref{a14} \eqref{a15} \eqref{a16} \eqref{a17} \eqref{a18} \eqref{a19} and \eqref{a20}, we are led to
\begin{equation}
\begin{split}
&\int_{\Omega}\left<
\begin{pmatrix}
\Delta U_{1,\mu}+\frac{1}{\varepsilon^2}e^{U_{2,\mu}+u_{0,2}}\\
\Delta U_{2,\mu}+\frac{1}{\varepsilon^2}e^{U_{1,\mu}+u_{0,1}}
\end{pmatrix}
,\begin{pmatrix}
Y_{1,0}\\
Y_{2,0}
\end{pmatrix}
\right>\\
=&-\frac{8}{\mu^3\rho_1^*} \left( \sum_{i=1}^k \rho_i^* \left( \int_{\Omega_i \setminus B_{\delta}(x_i)}  \frac{e^{f_{2,i,\textbf{x}}}-1}{|y-x_i|^4} - \int_{\mathbb{R}^2 \setminus \Omega_i} \frac{1}{|y-x_i|^4}    \right)     \right)\\
&-\frac{8}{\mu^3\rho_1} \left( \sum_{i=1}^k \rho_i \left( \int_{\Omega_i \setminus B_{\delta}(x_i)}  \frac{e^{f_{1,i,\textbf{x}}}-1}{|y-x_i|^4} - \int_{\mathbb{R}^2 \setminus \Omega_i} \frac{1}{|y-x_i|^4}    \right)     \right)\\
&+O\left( \frac{\delta^2}{\mu^3} + \frac{1}{\mu^4} \right)\\
\end{split}
\end{equation}
By the argument in Lemma \ref{lemma34},
\begin{equation}
\int_{\Omega} \left<  \begin{pmatrix}
 -\frac{1}{\varepsilon^2}e^{U_{1,\mu}+u_{0,1}+U_{2,\mu}+u_{0,2}} \\
 -\frac{1}{\varepsilon^2}e^{U_{1,\mu}+u_{0,1}+U_{2,\mu}+u_{0,2}}
\end{pmatrix}
\begin{pmatrix}
Y_{1,0}\\
Y_{2,0}
\end{pmatrix}
\right> =  B\varepsilon^2\mu+O(\frac{1}{\mu^5}),
\end{equation}
for some $B>0$, and

\begin{equation}
\int_{\Omega} \left<  \begin{pmatrix}
 \Delta \omega_{1,\mu}+\frac{1}{\varepsilon^2}R_{1,\mu}\\
 \Delta \omega_{2,\mu}+\frac{1}{\varepsilon^2}R_{2,\mu}
\end{pmatrix},
\begin{pmatrix}
Y_{1,0}\\
Y_{2,0}
\end{pmatrix}
\right>  =  O\left(\frac{\ln \mu}{\mu^{4-\frac{\alpha}{2}}}\right),
\end{equation}
where

$$
R_{1,\mu}=e^{U_{2,\mu}+u_{0,2}+\omega_{2,\mu}}(1-e^{U_{1,\mu}+u_{0,1}+\omega_{1,\mu}})-
e^{U_{2,\mu}+u_{0,2} }(1-e^{U_{1,\mu}+u_{0,1} })
$$
and
$$
R_{2,\mu}=e^{U_{1,\mu}+u_{0,1}+\omega_{1,\mu}}(1-e^{U_{2,\mu}+u_{0,2}+\omega_{2,\mu}})-
e^{U_{1,\mu}+u_{0,1} }(1-e^{U_{2,\mu}+u_{0,2} }).
$$
\end{proof}
{\bf Proof of Theorem \ref{main-thm}}
From \eqref{eq25} and \eqref{eq31}, we observe that \eqref{eq07} and \eqref{eq08} are equivalent to
\begin{equation}\label{eq53}
 DG_i^* ({\bf x})=O\left(\frac{\ln \mu}{\mu^{2-\frac{\alpha}{2}}}\right),\, i=1,2,
\end{equation}
and

\begin{equation}\label{eq54}
\begin{split}
&\frac{8}{\mu^3\rho_1}\left( \sum_{i=1}^k\rho_i(\int_{\Omega_i\setminus B_{\delta}(x_i)} \frac{e^{f_{1,i}}-1}{|y-x_i|^4}-\int_{\mathbb{R}^2\setminus\Omega_i}\frac{1}{|y-x_i|^4}         )\right)\\
&+\frac{8}{\mu^3\rho^*_1}\left( \sum_{i=1}^k\rho^*_i(\int_{\Omega_i\setminus B_{\delta}(x_i)} \frac{e^{f_{2,i}}-1}{|y-x_i|^4}-\int_{\mathbb{R}^2\setminus\Omega_i}\frac{1}{|y-x_i|^4}         )\right)
 +B\varepsilon^2\mu\\
=&\frac{1}{\mu^3}O\left( (|DG_1^*({\bf x})|^2+|DG_2^*({\bf  x})|^2)\ln\mu+\delta^2             \right)+O(\frac{1}{\mu^5})
\end{split}
\end{equation}

Since we assume  $\mathcal{D}^{(2)}({\bf q })<0$, for small $\delta>0$, there exists ${\bf x}$ close to ${\bf q}$
such that
\begin{equation}
\begin{split}
&  \sum_{i=1}^k\left( \frac{\rho_i}{\rho_1}(\int_{\Omega_i\setminus B_{\delta}(x_i)} \frac{e^{f_{1,i}}-1}{|y-x_i|^4}-\int_{\mathbb{R}^2\setminus\Omega_i}\frac{1}{|y-x_i|^4}         )
 +   \frac{\rho^*_i}{\rho^*_1}(\int_{\Omega_i\setminus B_{\delta}(x_i)} \frac{e^{f_{2,i}}-1}{|y-x_i|^4}-\int_{\mathbb{R}^2\setminus\Omega_i}\frac{1}{|y-x_i|^4} )\right)   \\
&+O(\delta^2)<0 .
\end{split}
\end{equation}
From this and  (A.2), we find that \eqref{eq53} and \eqref{eq54} have a solution
${\bf x_{\varepsilon}}$ and $(\mu_{1,\varepsilon},\cdots, \mu_{k,\varepsilon})$ satisfies
\begin{equation}
|DG_1^*({\bf x})|+|DG_2^*({\bf x})|\leq C \frac{\ln \mu}{\mu^{2-\frac{\alpha}{2}}},\,\, \mu_{i,\varepsilon}\in (\frac{\beta_0}{\sqrt{\varepsilon}},\frac{\beta_1}{\sqrt{\varepsilon}} ),\,\,i=1,\cdots,k.
\end{equation}

\section{Appendix}\label{Inv-sec}
In this section, we will discuss the invertibility of the linear operator $L_{\mu}$.  From our construction of the approximation solutions, we can  split  the associated  linear operator $L_{\mu}$  into two parts $\mathcal{L}_1$ and $\mathcal{L}_2$(see \eqref{eq58} below).  Apply  Theorem A.2 in \cite{LY2} and   Theorem B.1 in \cite{LY} to  $\mathcal{L}_1$ and $\mathcal{L}_2$ respectively,  we obtain  the invertibility of   $L_{\mu}$.

\begin{theorem}\label{inv-thm}
\begin{itemize}

\item[(1)] The operator $Q_{\mu}L_{\mu}$ is an isomorphism from $E_{\mu,2}$
to $F_{\mu,2}$. \\
\item[(2)] If  $\begin{pmatrix}\omega_{1}\\\omega_{2}  \end{pmatrix}\in E_{\mu,2}$ and
$\begin{pmatrix}h_{1}\\h_{2}  \end{pmatrix} \in F_{\mu,2}$  satisfies
\begin{equation}\label{eq50}
 L_{\mu}\begin{pmatrix}\omega_{1 }\\\omega_{2 }  \end{pmatrix}
 =\begin{pmatrix}h_{1 }\\h_{2 }  \end{pmatrix} .
\end{equation}
  Then there exits a constant $C>0$, independent of $\bf{x}$ and $\mu$,
  such that
\begin{equation}\label{eq40}
\left\Vert\begin{pmatrix}\omega_{1 }\\\omega_{2 }  \end{pmatrix}\right\Vert_{L^{\infty}(\Omega)}+\left\Vert\begin{pmatrix}\omega_{1 }\\\omega_{2 }  \end{pmatrix}\right\Vert_{X_{\alpha,\mu,2}}\leq C \ln \mu \left\Vert\begin{pmatrix}h_{1 }\\h_{2 }  \end{pmatrix} \right\Vert_{Y_{\alpha,\mu,2}}.
\end{equation}

\end{itemize}
\end{theorem}

\begin{proof}
Since $Z_{1,0}=Z_{2,0}$, $Z_{1,i,j}=Z_{2,i,j}$, $Y_{1,0}=Y_{2,0}$, and $Y_{1,i,j}=Y_{2,i,j}$, $i=1,2$, $j=1,\cdots,k.$
We denote
$$E_{\mu }=\left\{ \omega \in X_{\alpha,\mu }: \int_{\Omega}
\omega
Z_{1,0}
 = \int_{\Omega}
\omega
Z_{1,i,j} =0,\, i=1,2,\, j=1,\cdots, k.    \right\}
$$

and
$$F_{\mu }=\left\{
\omega \in Y_{\alpha,\mu }  : \int_{\Omega}
\omega
Y_{1,0} = \int_{\Omega}
\omega_1
Y_{1,i,j} =0,\, i=1,2,\, j=1,\cdots, k.    \right\}.
$$
We use the same notation for the project operator $Q_{\mu}:Y_{\mu}\to F_{\mu}$.

We can   rewrite  \eqref{eq50}  we as
\begin{equation}\label{eq58}
\begin{split}
\mathcal{L}_1( \omega_{1}+\omega_{2} ):=&\Delta(\omega_1+\omega_2)+ \sum_{i=1}^k1_{B_{d_i}(x_i)} e^{V_{x_i,\mu_i}}(\omega_{1}+\omega_{2})=h_1+h_2\\
\mathcal{L}_2( \omega_{1}-\omega_{2} ):=&\Delta(\omega_1-\omega_2) - \sum_{i=1}^k1_{B_{d_i}(x_i)} e^{V_{x_i,\mu_i}}(\omega_{1}-\omega_{2})=h_1-h_2
\end{split}
\end{equation}

By  Theorem A.2 in \cite{LY2},  $Q_{\mu}\mathcal{L}_1$ is an isomorphism from $E_{\mu}$ to $F_{\mu}$; by    Theorem B.1 in \cite{LY},  $\mathcal{L}_2$ is an isomorphism from $X_{\alpha,\mu}$ to $Y_{\alpha,\mu}$.  Furthermore,
by Theorem A.2 in \cite{LY2}  and   Theorem B.1 in \cite{LY} again,     we obtain
\begin{equation}
\left\Vert \omega_{1 }+\omega_2 \right\Vert_{L^{\infty}(\Omega)}+\left\Vert\omega_{1 }+\omega_2 \right\Vert_{X_{\alpha,\mu}}\leq C \ln \mu \left\Vert h_{1 }+h_{2 }   \right\Vert_{Y_{\alpha,\mu }}
\end{equation}
and
\begin{equation}
\left\Vert \omega_{1 }-\omega_2 \right\Vert_{L^{\infty}(\Omega)}+\left\Vert\omega_{1 }-\omega_2 \right\Vert_{X_{\alpha,\mu}}\leq C \ln \mu \left\Vert h_{1 }-h_{2}   \right\Vert_{Y_{\alpha,\mu }}
\end{equation}
which imply \eqref{eq40}.

\end{proof}

\bibliography{nsf}

\begin{thebibliography}{10}

\bibitem{abjm}
O.~Aharony, O.~Bergman, D.~Jafferis, and J.~Maldacena.
\newblock {$N=6$} superconformal {C}hern-{S}imons-matter theories, {M}2-branes
  and their gravity duals.
\newblock {\em J. High Energy Phys.}, (10):091, 38, 2008.

\bibitem{Bo}
E.~B. Bogomolny.
\newblock The stability of classical solutions.
\newblock {\em Jadernaja Fiz.}, 24(4):861--870, 1976.

\bibitem{CY1}
L.~Caffarelli and Y.~Yang.
\newblock Vortex condensation in the {C}hern-{S}imons {H}iggs model: an
  existence theorem.
\newblock {\em Comm. Math. Phys.}, 168(2):321--336, 1995.

\bibitem{CFL}
H.~Chan, C.~Fu, and C.-S. Lin.
\newblock Non-topological multi-vortex solutions to the self-dual
  {C}hern-{S}imons-{H}iggs equation.
\newblock {\em Comm. Math. Phys.}, 231(2):189--221, 2002.

\bibitem{CCL}
J.~Chern, Z.~Chen, and C.S. Lin.
\newblock Uniqueness of topological solutions and the structure of solutions
  for the {C}hern-{S}imons system with two {H}iggs particles.
\newblock {\em Comm. Math. Phys.}, 296(2):323--351, 2010.

\bibitem{CY2014}
J.~Chern and S.~Yang.
\newblock The non-topological fluxes of a two-particle system in the
  {C}hern-{S}imons theory.
\newblock {\em J. Differential Equations}, 256(10):3417--3439, 2014.

\bibitem{C}
K.~Choe.
\newblock Uniqueness of the topological multivortex solution in the self-dual
  {C}hern-{S}imons theory.
\newblock {\em J. Math. Phys.}, 46(1):012305, 22, 2005.

\bibitem{choe}
K.~Choe.
\newblock Multiple existence results for the self-dual {C}hern-{S}imons-{H}iggs
  vortex equation.
\newblock {\em Comm. Partial Differential Equations}, 34(10-12):1465--1507,
  2009.

\bibitem{CK}
K.~Choe and N.~Kim.
\newblock Blow-up solutions of the self-dual {C}hern-{S}imons-{H}iggs vortex
  equation.
\newblock {\em Ann. Inst. H. Poincar\'e Anal. Non Lin\'eaire}, 25(2):313--338,
  2008.

\bibitem{CKL}
K.~Choe, N.~Kim, and C.S. Lin.
\newblock Existence of self-dual non-topological solutions in the
  {C}hern-{S}imons {H}iggs model.
\newblock {\em Ann. Inst. H. Poincar\'e Anal. Non Lin\'eaire}, 28(6):837--852,
  2011.

\bibitem{DEFM}
M.~Del~Pino, P.~Esposito, P.~Figueroa, and M.~Musso.
\newblock Nontopological condensates for the self-dual {C}hern-{S}imons-{H}iggs
  model.
\newblock {\em Communications on Pure and Applied Mathematics}, 2014.

\bibitem{Dunne1}
G.~V. Dunne.
\newblock Aspects of {C}hern-{S}imons theory.
\newblock In {\em Aspects topologiques de la physique en basse
  dimension/{T}opological aspects of low dimensional systems ({L}es {H}ouches,
  1998)}, pages 177--263. EDP Sci., Les Ulis, 1999.

\bibitem{dz94}
J.~Dziarmaga.
\newblock Low energy dynamics of ${[\mathrm{U}(1)]}^{N}$ chern-simons solitons.
\newblock {\em Phys. Rev. D}, 49:5469--5479, May 1994.

\bibitem{Gu1}
S.~Gudnason.
\newblock Non-abelian {C}hern-{S}imons vortices with generic gauge groups.
\newblock {\em Nuclear Phys. B}, 821(1-2):151--169, 2009.

\bibitem{Gu2}
S.~Gudnason.
\newblock Fractional and semi-local non-{A}belian {C}hern-{S}imons vortices.
\newblock {\em Nuclear Phys. B}, 840(1-2):160--185, 2010.

\bibitem{Hagen}
C.R. Hagen.
\newblock {Parity conservation in Chern-Simons theories and the anyon
  interpretation}.
\newblock {\em Phys.Rev.Lett.}, 68:3821--3825, 1992.

\bibitem{HHL1}
X.~Han, H.~Huang, and C.-S. Lin.
\newblock Bubbling solutions for a skew-symmetric chern–simons system in a
  torus.
\newblock {\em Journal of Functional Analysis}, 273(4):1354 -- 1396, 2017.

\bibitem{HKP}
J.~Hong, Y.~Kim, and P.~Pac.
\newblock Multivortex solutions of the abelian {C}hern-{S}imons-{H}iggs theory.
\newblock {\em Phys. Rev. Lett.}, 64(19):2230--2233, 1990.

\bibitem{HL1}
H.~Huang and C.-S. Lin.
\newblock Uniqueness of non-topological solutions for the {C}hern-{S}imons
  system with two {H}iggs particles.
\newblock {\em Kodai Math. J.}, 37(2):274--284, 2014.

\bibitem{HZ-2015}
H.~Huang and L.~Zhang.
\newblock The domain geometry and the bubbling phenomenon of rank two gauge
  theory.
\newblock {\em Comm. Math. Phys.}, 349(1):393--424, 2017.

\bibitem{JW}
R.~Jackiw and E.~Weinberg.
\newblock Self-dual {C}hern-{S}imons vortices.
\newblock {\em Phys. Rev. Lett.}, 64:2234--2237, May 1990.

\bibitem{JT}
A.~Jaffe and C.~Taubes.
\newblock {\em Vortices and monopoles}, volume~2 of {\em Progress in Physics}.
\newblock Birkh\"{a}user, Boston, Mass., 1980.
\newblock Structure of static gauge theories.

\bibitem{KLKLM}
C.~Kim, C.~Lee, P.~Ko, B.~Lee, and H.~Min.
\newblock Schr\"odinger fields on the plane with ${[\mathrm{U}(1)]}^{N}$
  {C}hern-{S}imons interactions and generalized self-dual solitons.
\newblock {\em Phys. Rev. D}, 48:1821--1840, Aug 1993.

\bibitem{LPY}
C.-S. Lin, A.~Ponce, and Y.~Yang.
\newblock A system of elliptic equations arising in {C}hern-{S}imons field
  theory.
\newblock {\em Journal of Functional Analysis}, 247(2):289 -- 350, 2007.

\bibitem{LP}
C.-S. Lin and J.~Prajapat.
\newblock Vortex condensates for relativistic abelian {C}hern-{S}imons model
  with two {H}iggs scalar fields and two gauge fields on a torus.
\newblock {\em Comm. Math. Phys.}, 288(1):311--347, 2009.

\bibitem{linwang2010}
C.-S. Lin and C.~Wang.
\newblock {Elliptic functions, {G}reen functions and the mean field equations
  on tori.}
\newblock {\em {Ann. Math. (2)}}, 172(2):911--954, 2010.

\bibitem{linwang20171}
C.-S. Lin and C.~Wang.
\newblock On the minimality of extra critical points of {G}reen functions on
  flat tori.
\newblock {\em Int. Math. Res. Not. IMRN}, (18):5591--5608, 2017.

\bibitem{LY}
C.-S. Lin and S.~Yan.
\newblock Bubbling solutions for relativistic abelian {C}hern-{S}imons model on
  a torus.
\newblock {\em Comm. Math. Phys.}, 297(3):733--758, 2010.

\bibitem{LY2}
C.-S. Lin and S.~Yan.
\newblock Existence of bubbling solutions for {C}hern-{S}imons model on a
  torus.
\newblock {\em Arch. Ration. Mech. Anal.}, 207(2):353--392, 2013.

\bibitem{LY3}
C.-S. Lin and S.~Yan.
\newblock On condensate of solutions for the {C}hern–{S}imons–{H}iggs
  equation.
\newblock {\em Annales de l'Institut Henri Poincare (C) Non Linear Analysis},
  34(5):1329 -- 1354, 2017.

\bibitem{ZL-JFA}
C.-S. Lin and L.~Zhang.
\newblock On liouville systems at critical parameters, part 1: One bubble.
\newblock {\em Journal of Functional Analysis}, 264(11):2584 -- 2636, 2013.

\bibitem{LMMS}
G.~S. Lozano, D.~Marqu\'{e}s, E.~F. Moreno, and F.~A. Schaposnik.
\newblock Non-abelian {C}hern-{S}imons vortices.
\newblock {\em Phys. Lett. B}, 654(1-2):27--34, 2007.

\bibitem{PS}
M.~K. Prasad and C.~Sommerfield.
\newblock Exact classical solution for the 't hooft monopole and the julia-zee
  dyon.
\newblock {\em Phys. Rev. Lett.}, 35:760--762, Sep 1975.

\bibitem{SFEG}
S.~Spielman, K.~Fesler, C.~B. Eom, T.~H. Geballe, M.~M. Fejer, and
  A.~Kapitulnik.
\newblock Test for nonreciprocal circular birefringence in
  ${\mathrm{yba}}_{2}$${\mathrm{cu}}_{3}$${\mathrm{o}}_{7}$ thin films as
  evidence for broken time-reversal symmetry.
\newblock {\em Phys. Rev. Lett.}, 65:123--126, Jul 1990.

\bibitem{SY}
J.~Spruck and Y.~Yang.
\newblock The existence of nontopological solitons in the self-dual
  {C}hern-{S}imons theory.
\newblock {\em Comm. Math. Phys.}, 149(2):361--376, 1992.

\bibitem{T1}
G.~Tarantello.
\newblock Multiple condensate solutions for the {C}hern-{S}imons-{H}iggs
  theory.
\newblock {\em J. Math. Phys.}, 37(8):3769--3796, 1996.

\bibitem{T}
G.~Tarantello.
\newblock Uniqueness of selfdual periodic {C}hern-{S}imons vortices of
  topological-type.
\newblock {\em Calc. Var. Partial Differential Equations}, 29(2):191--217,
  2007.

\bibitem{Tbook}
G.~Tarantello.
\newblock {\em Selfdual gauge field vortices}.
\newblock Progress in Nonlinear Differential Equations and their Applications,
  72. Birkh\"auser Boston, Inc., Boston, MA, 2008.
\newblock An analytical approach.

\bibitem{Wil}
F.~Wilczek.
\newblock Disassembling anyons.
\newblock {\em Phys. Rev. Lett.}, 69:132--135, Jul 1992.

\bibitem{yangbook}
Y.~Yang.
\newblock {\em Solitons in field theory and nonlinear analysis}.
\newblock Springer Monographs in Mathematics. Springer-Verlag, New York, 2001.

\end{thebibliography}
\bibliographystyle{plain}

\end{document}